\documentclass[11pt]{amsart}
\usepackage{tikz}
\usepackage{pgfplots}

\usepackage{amssymb}
\usepackage{amsthm}
\usepackage{epsfig}
\usepackage{xcolor}
\usepackage{verbatim}

\usepackage{amsmath}
\usepackage{hyperref}
\usepackage{xypic}
\usepackage{amscd}
\usepackage{color}
\newfont{\sheaf}{eusm10 scaled\magstep1}

\newtheorem{thm}{Theorem}[section]

\newtheorem{lemma}[thm]{Lemma}
\newtheorem{prop}[thm]{Proposition}

\theoremstyle{definition}

\newtheorem{definition}[thm]{Definition}
 \newtheorem{notation}[thm]{Notation}

\newcommand{\sslash}{\mathbin{/\mkern-6mu/}}

\DeclareMathOperator{\Hilb}{Hilb}
\DeclareMathOperator{\PGL}{PGL}
\DeclareMathOperator{\Ker}{Ker}

\DeclareMathOperator{\Sing}{Sing}
\DeclareMathOperator{\Aut}{Aut}

\DeclareMathOperator{\Pic}{Pic}

\DeclareMathOperator{\Sec}{Sec}

\def\c1{\operatorname{c_1}}
\def\c2{\operatorname{c_2}}

 \def\Hilb{\operatorname{Hilb}}

\def\Sym{\operatorname{Sym}}

\def\n{\mathbf{N}}

\def\c{\mathfrak{P}}
\def\CC{{\mathbb C}}
\def\ZZ{{\mathbb Z}}

\def\PP{{\mathbb P}}

\def\MM{{\mathbb M}}
\def\A{{\mathcal A}}

\def\R{{\mathcal R}}

\def\M{{\mathcal M}}

\def\O{{\mathcal O}}
\def\I{{\mathcal I}}

\def\Z{{\mathcal Z}}
\def\E{{\mathcal E}}

\def\F{{\mathcal F}}

\def\P{{\mathcal P}}

\def\GG{{\mathbb G}}
\def\x{\times}                   % product (fiber)
\def\cong{\simeq}

\def\+{\oplus}                   % direct sum
\def\*{\otimes}                  % tensor product
\def\mod{\operatorname{mod}}
\def\Aut{\operatorname{Aut}}

\def\Pic{\operatorname{Pic}}

\def\Sing{\operatorname{Sing}}

\begin{document}

\title[]{(Uni)rational parametrizations of  $\mathcal R_{g,2}$, $\mathcal R_{g,4}$ and $\mathcal R_{g,6}$ in low genera}

\author[A.~L.~Knutsen]{Andreas Leopold Knutsen}
\address{A.~L.~Knutsen, Department of Mathematics, University of Bergen,
Postboks 7800,
5020 Bergen, Norway}
\email{andreas.knutsen@math.uib.no}

\author[M.~Lelli-Chiesa]{Margherita Lelli-Chiesa}
\address{M.~ Lelli-Chiesa, Dipartimento di Matematica,
Universit{\`a} Roma Tre,
Largo San Leonardo Murialdo,  
00146 Roma,  Italy}
\email{margherita.lellichiesa@uniroma3.it} 

\author[A.~Verra]{Alessandro Verra}
\address{A.~ Verra, 
Dipartimento di Matematica,
Universit{\`a} Roma Tre,
Largo San Leonardo Murialdo,  
00146 Roma,  Italy} \email{verra@mat.uniroma3.it}
\begin{abstract}
The moduli space $\R_{g,2n}$ parametrizes double covers of smooth curves of genus $g$ ramified at $2n$ points. We will prove the (uni)rationality of $\R_{g,2}$, $\mathcal R_{g,4}$ and $\mathcal R_{g,6}$ in low genera.
\end{abstract}
\maketitle
\section{Introduction}
The study of finite covers of complex algebraic curves dates back to Riemann  \cite{Ri}, whose proof of the fact that compact Riemann surfaces of genus $g\geq 2$  depend on $3g-3$ parameters relies on the existence of a degree $g+1$ map from any such surface to the Riemann sphere. The particular interest in double covers has been ever-growing in the last fifty years since Mumford's algebraic construction of the Prym variety \cite{Mu}, an abelian variety which is naturally associated \color{black} to any such cover. The \'etale case seized most of the attention mainly because the Prym variety of an \'etale double cover is principally polarized.  However, in the last decade several mathematicians got engaged in  Prym varieties of ramified double covers \cite{MP,NOV,NO}, and this is a good impulse for studying their moduli spaces.

A double cover of curves $\pi:\widetilde C\to C$ is, by standard theory, equivalent to the datum of an effective divisor $x_1+\cdots+x_{2n}\in \Sym^{2n}(C)$ (corresponding to the branch divisor of $\pi$) and of a line bundle $\eta\in \Pic^{-n}(C)$ such that $\eta^{\otimes 2}\simeq \mathcal O_C(-x_1-\ldots-x_{2n})$. Double covers of genus $g$ curves ramified at $2n$ points are thus parametrized by the following moduli space 
\begin{equation*}
\begin{split}
\mathcal R_{g,2n}:=\left\{(C,x_1+\ldots+x_{2n},\eta)\,|\,\right. [C]\in\M_g,\;&x_i\in C\,\forall\, 1\leq i\leq n\\
&\left.\,\eta\in\Pic^{-n}(C),\; \eta^2=\mathcal O_C(-x_1-\ldots-x_{2n}) \right\}.
\end{split}
\end{equation*}
For $g=0$ one recovers the space $\mathcal H_{n+2}$ of hyperelliptic curves of genus $n-1$, that was proved to be rational for any $n$ \cite{Bo,K1,K2}.  In the case $g=1$, the space $\R_{1,2n}$ is birational to the moduli space of bielliptic curves of genus $n+1$, which is known to be rational for $3\leq g\leq 5$ \cite{BdC1,BdC2,CdC1} and unirational for $g\geq 6$ \cite{CdC2}.  

For $g\geq 2$, only the \'etale case (that is, $n=0$) has been largely investigated: in this case the moduli space is denoted by $\R_g$ and its points parametrize Prym curves, that is, pairs $(C,\eta)$, where $C$ is a smooth curve of genus $g$ and $\eta$ a non-trivial $2$-torsion line bundle on it. The moduli space $\R_g$ is known to be rational for $2\leq g\leq 4$ \cite{Do2,Ca}, unirational for $5\leq g\leq 7$ \cite{Don,ILS,V1,V2,FV1,FV2}, and uniruled for $g=8$ \cite{FV2}. On the other hand, it is of general type for $g\geq 13$ and $g\neq 16$ \cite{FL,Br}. 

As soon as $g\geq 2$ and $n\geq 1$, the geometry of $\mathcal R_{g,2n}$ becomes almost unexplored. Only recently, Bud \cite{Bu} proved the irreducibility of $\mathcal R_{g,2n}$ and initiated the study of its birational geometry when $n=1$, obtaining that $\mathcal R_{g,2}$ is of general type for $g\geq 16$ and uniruled for $3\leq g\leq 6$. Up to our knowledge, the only further result in the literature is due to the third named author \cite{NOV} and concerns the rationality of $\R_{2,6}$. This paper proves the following rationality/unirationality results for $\mathcal R_{g,2n}$ in the cases $n=1,2,3$ for low values of $g$:

\begin{thm}\label{pescara}
\begin{itemize}
\item[(a)] The moduli space $\R_{g,2}$ is unirational for $3\leq g\leq 5$, while possesses a unirational divisor for $g=6$. \color{black}
\item[(b)]  The moduli space $\R_{g,4}$ is unirational for $2\leq g\leq 5$.
\item[(c)] The moduli space $\R_{g,6}$ is rational for $g=2$ and  unirational for $g=3$. \color{black}
\end{itemize}
\end{thm}

This is obtained by exhibiting (uni)rational parametrizations of the above moduli spaces using Nikulin surfaces of non-standard type, but in the case of $\mathcal R_{6,2}$.\par  \begin{comment}In the latter two cases, somehow interestingly, the type of rational map constructed in all the other cases fails to be dominant over the target moduli space. Nevertheless we employ and extend, for this case,  the method adopted in \cite{FV2} to prove the unirationality of $\mathcal R_7$. It turns out that this method, relying on standard Nikulin surfaces and working "ad hoc" for the case of $\mathcal R_{6,2}$, makes possible to prove the rationality of $\mathcal R_{6,2}$. \end{comment}

 A Nikulin surface of genus $h\geq 2$ is a polarized $K3$ surface $(S,H)$ endowed with an \'etale double cover $\pi$ branched along eight irreducible ($-2$)-curves $N_1,\ldots,N_8$, which are disjoint both pairwise and from $H$; the line bundle $M$ defining the double cover $\pi$ thus satisfies $M\sim N_1+\cdots+N_8$. A Nikulin surface is called of non-standard type if $h$ is odd and the lattice generated by $H,N_1,\ldots,N_8,M$ has index $2$ in $\Pic(S)$. If this occurs, then $H(-M)\sim R+R'$ where $|R|$ and $|R'|$ are base point free linear systems on $S$ \cite{GS} such that the restriction of $\pi$ to the inverse image of a general curve $C$ in either of them defines a branched cover of $C$ ramified at $2$, $4$ or $6$ points depending on the congruence of $h$ modulo $4$ and on whether $C\in |R|$ or $C\in |R'|$.  Furthermore, by varying $h$, one obtains all possible values for the genus $g$ of $C$. This fact was exploited for the first time in \cite{DL} to show that, if $\widetilde C\to C$ is a general double cover of a genus $g$ curve ramified at $2$, $4$ or $6$, then $\widetilde C$  is Brill-Noether general. 
For dimensional reasons, one may hope for a general element of $\mathcal R_{g,2n}$ to lie on a Nikulin surface of non-standard type whenever $g\leq 6$ and $n=1$, or $g\leq 5$ and $n=2$ and finally for $g\leq 4$ if $n=3$. The proof of Theorem \ref{pescara} consists in verifying this expectation, except when $g=6$ and $n=1$, or $g=4$ and $n=3$. The space $\R_{g,2n}$ is thus dominated by a projective bundle $\P$ over the moduli space of Nikulin surfaces of genus $h$ and non-standard type, that is denoted by  $\F_h^{\n,ns}$, and one reduces to proving (uni)rationality of the latter:
\begin{thm}\label{moduli}
The moduli space $\F_h^{\n,ns}$ is rational for $h=23,11$ and unirational for $h=21,19,17, 15,13,9$.
\end{thm}
We stress that the cases $h=9,11$ had already been established  by the same authors in \cite{KLV2}. We also recall that the other type of Nikulin surfaces, which are called standard, were used in \cite{FV1} to provide unirational parametrizations of the Prym moduli space $\R_g$ in low genera.
It is unknown whether for high enough values of $h$ the moduli space $\F_h^{\n,ns}$ and its counterpart in the standard case become of general type.

In our construction we consider the projective model of a general Nikulin surface of non-standard type defined by either $|R|$ or $|R'|$ and find in the other linear system $|R'|$ or $|R|$, respectively, a reducible curve consisting of some lines and a distinguished component, which is either a rational  or elliptic normal curve. This component is at the core of our parametrizations.

The paper is organized as follows. In section \ref{sec:def} we recall the definition and main features of Nikulin surfaces of standard and non-standard type and introduce relevant moduli maps from some bundles on $\F_h^{\n,ns}$ to the moduli space $\R_{g,2n}$ for $n=1,2,3$. Section \ref{figo} contains results that relate the space of quadrics in $\PP^n$ containing either a rational or an elliptic normal curve. The proofs of Theorems \ref{pescara} and \ref{moduli} is the content of the remaining sections, each of which contains the results related to $\F_h^{\n,ns}$ for a specific value of $h$. Throughout the paper the irreducibility of $\F_h^{\n,ns}$ is constantly used; however, when $h \equiv 1 \; \mod 4$ the numerical properties of $R$ and $R'$ are the same and so there is no way to distinguish among them. We therefore introduce a double cover $\widehat{\F}_h^{\n,ns}$ of $\F_h^{\n,ns}$ that parametrizes Nikulin surfaces marked with a line bundle having the numerical properties of $R$ and $R'$. The irreducibility of $\widehat{\F}_h^{\n,ns}$ is then established in Appendix \ref{appendix}.

\section{Nikulin surfaces of non-standard type and curves in $\R_{g,2}$, $\R_{g,4}$ and $\R_{g,6}$} \label{sec:def}

We recall some basic definitions and properties.
\begin{definition} \label{def:Nik}
  A primitively polarized  {\it Nikulin} surface of genus $h \geq 2$ is a triple $(S,M,H)$ such that 
$(S,H)$ is a primitively genus $h$ polarized $K3$ surface and $\O_S(M)\in \Pic S$ satisfies the following conditions:
\begin{itemize}
\item $ N_1+\cdots+N_8 \sim 2M$ for $8$ pairwise disjoint ($-2$)-curves $N_1,\ldots,N_8$;
\item $H \cdot M=0$.
\end{itemize}
\end{definition}
The Picard group of a Nikulin surface $S$ always contains the rank $8$ sublattice of $\Pic S$ generated by $N_1,\ldots,N_8$ and $M$, which is called  {\it Nikulin lattice} and is denoted by $\mathbf{N}$.

Since $H\cdot M=0$,  we get a rank $9$ sublattice
\[ \Lambda_h:= \ZZ[H] \+_{\perp} \mathbf{N} \subset \Pic S.\]

A Nikulin surface $(S,M,H)$ is said {\it of standard type} if the embedding $\Lambda_h \subset \Pic S$ is primitive,  and of {\it non-standard type} otherwise. By \cite[Prop. 2.2]{vGS}, in the non-standard case the genus $h$ is odd and the embedding $\Lambda_h \subset \Pic S$ has index $2$. More precisely (cf. \cite[Prop. 2.1 and Cor. 2.1]{GS}), the surface $S$ carries two line bundles $R$ and $R'$ whose class, up to renumbering the curves $N_i$, can be written as follows:
 $$R\sim \frac{H-N_1-N_2}{2},\;\;R'\sim\frac{H-N_3-\cdots-N_8 }{2}\textrm{  if  } h \equiv 3 \; \mod 4;$$ 
 $$R\sim \frac{H-N_1-N_2-N_3-N_4}{2},\;\; R'\sim \frac{H-N_5-N_6-N_7-N_8}{2}\textrm{  if  } h \equiv 1 \; \mod 4.$$
Setting $g:=g(R)$ and $g':=g(R')$, we have $g=(h+1)/4$ and $g'=(h-3)/4$ in the former case, while $g=g'=(h-1)/4$ in the latter one.

We denote by $\F_h^{\n,ns}$  the coarse moduli spaces of genus $h$ primitively polarized Nikulin surfaces of non-standard type, which is irreducible of dimension $11$ by \cite[\S 3]{Do}. 

By \cite[Prop. 2.1 and Cor. 2.1]{GS}, a very general $(S,M,H)$ in $\F_h^{\n,ns}$ has Picard number $9$ and
\begin{equation}
  \label{eq:pic9}
\Pic S \cong \ZZ[R] \+ \mathbf{N}.
\end{equation}
Furthermore, as soon as $h\geq 5$, both $R$ and $R'$ are globally generated and satisfy $h^1(R)=h^1(R')=0$ (cf. \cite[Prop. 2.3(i)]{KLV2}). 

Let $C\in |R|$ be a smooth irreducible curve and look at the restriction $M|_C$, which is a line bundle of degree $1$ or $2$ depending on the congruence of $h$ modulo $4$. More precisely, setting $\{x_i\}:=C\cap N_i$, we have $M|_C^{\otimes 2}\simeq \mathcal O_C(x_1+x_2)$ if  $h \equiv 3 \; \mod 4$ and $M|_C^{\otimes 2}\simeq \mathcal O_C(x_1+x_2+x_3+x_4)$ if  $h \equiv 1 \; \mod 4$. In the former case the triple $(C,x_1+x_2,M^\vee|_C)$ thus defines a point of $\mathcal R_{g,2}$, while $(C,x_1+x_2+x_3+x_4,M^\vee|_C)\in\mathcal R_{g,4}$ in the latter case. Analogously, any smooth irreducible curve $C'\in |R'|$ defines either a point $(C', x_3+x_4+x_5+x_6+x_7+x_8,M^\vee|_{C'})\in \mathcal R_{g',6}$ if $h \equiv 3 \; \mod 4$, or a point $(C',x_5+x_6+x_7+x_8,M^\vee|_{C'})\in\R_{g',4}$ otherwise. 

We now focus on the case $h \equiv 3 \; \mod 4$, and denote by $\P_{g}$ (respectively, $\P'_{g'}$) the moduli space of $4$-tuples $(S,M,H,C)$ where $(S,M,H)\in \F_h^{\n,ns}$ and $C\in |R|$ (resp., $C\in |R'|$) is a smooth irreducible curve. There are natural maps
\begin{equation}\label{modulimap}
\xymatrix{ 
&\P'_{g'}\ar[ld]_{r_{g',6}}     \ar[rd]^{q'_{g'}}&&\P_{g} \ar[ld]_{q_{g}}     \ar[rd]^{r_{g,2}} &\\
\R_{g',6}&&\F_h^{\n,ns}&     &\R_{g,2},
}
\end{equation}
where  $q_{g}$ and $q'_{g'}$ are the natural forgetful morphisms, while $r_{g,2}$ sends a point $(S,M,H,C)\in \P_{g}$ to $(C,x_1+x_2,M^\vee|_C)\in \mathcal R_{g,2}$, and $r_{g',6}$ is defined analogously.

In the case $h \equiv 1 \; \mod 4$, the line bundles $R,R'\in \Pic(S)$ have the same numerical properties and thus cannot be distinguished. We will therefore consider the moduli space $\widehat{\F}_h^{\n,ns}$ of quadruples $(S,M,H,R)$ such that $(S,M,H) \in \F_h^{\n,ns}$ and $R$ is a line bundle on $S$ such that $H-2R$ is the sum of four of the eight $(-2)$-curves in the Nikulin lattice. We will call such quadruples {\it marked non-standard Nikulin surfaces}. 
The forgetful map $\widehat{\F}_h^{\n,ns} \to \F_h^{\n,ns}$ realizes $\widehat{\F}_h^{\n,ns}$ as a double cover of $\F_h^{\n,ns}$, which turns out to be irreducible (cf. Appendix \ref{appendix}). As before, the moduli space ${\widehat\P}_{g}$ parametrizing $5$-tuples $(S,M,H,R,C)$, with $(S,M,H,R)\in \widehat{\F}_h^{\n,ns}$ and $C\in |R|$, admits two natural maps:
\begin{equation}\label{modulimap}
\xymatrix{ 
&{\widehat\P}_{g} \ar[ld]_{\tilde q_{g}}     \ar[rd]^{r_{g,4}} &\\
\widehat{\F}_h^{\n,ns}&     &\R_{g,4}.
}
\end{equation}
Since $\mathrm{dim}\,\mathcal R_{g,2n}=3g-3+2n$, $\dim\,\mathcal F_h^{\mathbf N,ns} =11$ and a genus $g$ linear system on a $K3$ surface has dimension $g$, the map $r_{g,2n}$ is expected to be dominant for $g\leq 6$ if $n=1$, for $g\leq 5$ when $n=2$ and for $g\leq 4$ if $n=3$.

In our proofs, the morphisms 
$$\varphi_{R'}:S\longrightarrow \overline{S'}\subset \PP(H^0(R')^\vee)$$
and 
$$\varphi_{R}:S\longrightarrow \overline{S}\subset \PP(H^0(R)^\vee)$$
will play a crucial role. The morphism $\varphi_{R'}$ contracts the curves $N_i$ such that $N_i\cdot R'=0$ to nodes $x_i$ of $\overline{S'}$ and maps the remaining $N_i$ to lines. The analogue holds for $\varphi_{R}$. We stress that the restriction of $\varphi_{R'}$ (respectively, $\varphi_{R}$) to any irreducible curve $C\in |R|$ (resp., $C\in |R'|$) is induced by the complete linear system $\vert \omega_C\otimes M\vert$.

\begin{notation}\label{pinguini}
Given any morphism between surfaces whose restriction to a curve $D$ is an isomorphism, we will denote the image curve still by $D$. 
\end{notation}

\section{Quadrics containing particular curves}\label{figo}

\subsection{Quadrics containing a rational normal curve}\label{federico}

Let $\Gamma\subset \PP^n$ be a degree $n$ rational normal curve and let $\Sec\Gamma\subset \PP^n$ be its secant variety, which has degree $n-1\choose{2}$. We consider the incidence variety 
$$
I_\Gamma:=\left\{(z,x+y)\in \Sec\Gamma\times \Sym^2\Gamma\,\,|\,\, z\in \langle x,y\rangle   \right\},
$$
along with the two projections $p:I_\Gamma\to \Sec\Gamma$ and $q: I_\Gamma\to \Sym^2\Gamma$. The map $p$ is birational and contracts the locus 
$$
\mathrm{Exc} (p)=\left\{(x,x+y)\in \Sec\Gamma\times \Sym^2\Gamma\,\,|\,\, x,y\in \Gamma   \right\}\simeq \Gamma\times \Gamma,
$$
to $\Gamma\subset \Sec\Gamma$. On the other hand, $q$ gives $I$ the structure of a $\PP^1$-bundle over $\Sym^2\Gamma$. 

From now on, we identify $\Sym^2(\Gamma)$ with $\PP^2$ by choosing an embedding of $\Gamma$ as a conic $\Delta\subset \PP^2$ and identifying a divisor $x+y\in \Sym^2(\Gamma)=\Sym^2(\Delta)$ with the pole of the line $\langle x,y\rangle$ with respect to $\Delta$ if $x\neq y$, and with the point $x\in \Delta$ if $x=y$. In this way, the diagonal of $\Sym^2(\Gamma)$ gets identified with $\Delta$, and for any $x\in \Gamma$ the curve $\sigma_x:=\{x+y\in \Sym^2(\Gamma)\,|\,y\in \Gamma\}$ is identified with the tangent line to $\Delta$ at $x$. 

Let now $|\I_{\Gamma/\PP^n}(2)|$ denote the space of quadrics containing $\Gamma\subset \PP^n$. 
\begin{prop}\label{razionale}
\begin{itemize}
\item[(i)] There exists an isomorphism 
$$
\alpha_\Gamma: |\I_{\Gamma/\PP^n}(2)|\longrightarrow |\mathcal O_{\PP^2}(n-2)|
$$ 
defined by setting $\alpha_\Gamma(Q):=q(p^{-1}(Q\cap\Sec \Gamma))$ for any $Q\in |\I_{\Gamma/\PP^n}(2)|$.
\item[(ii)] A quadric $Q\in |\I_{\Gamma/\PP^n}(2)|$ is singular at a point $x\in \Gamma$ if and only if $\alpha_\Gamma(Q)$ contains $\sigma _x$.

\end{itemize}
\end{prop}
\begin{proof}
We first show that the map $\alpha_\Gamma$ is well-defined. By \cite{Ca} $\Sec\Gamma$ is not contained in any quadric. \color{black}For any $Q\in |\I_{\Gamma/\PP^n}(2)|$, the intersection $Q\cap \Sec\Gamma$ is union of bisecant lines to $\Gamma$. Indeed, for any $z\in (Q\cap \Sec \Gamma)\setminus \Gamma$ there exists $x+y\in \Sym^2\Gamma$ such that $z\in \langle x,y\rangle$; the quadric $Q$ contains $x,y,z$ and thus the whole line $\langle x,y\rangle$. This shows that $\alpha_\Gamma (Q)=q(p^{-1}(Q\cap\Sec \Gamma))$ is indeed a curve in $\Sym^2(\Gamma)=\PP^2$. Its degree is $n-2$ because, by degree reasons, $Q$ intersects the rational normal scroll of degree $n-1$ in $\PP^n$ spanned by a $g^1_2$ on $\Gamma$ along $\Gamma$ and $n-2$ bisecant lines. Since $|\I_{\Gamma/\PP^n}(2)|$ and $|\mathcal O_{\PP^2}(n-2)|$ are projective spaces of the same dimension and $\alpha_\Gamma$ is injective by construction, it is an isomorphism. 

As concerns (ii), a quadric $Q$ is singular at a point $x\in \Gamma$ if and only if it contains all the secant lines to $\Gamma$ passing through $x$; this is equivalent to requiring that $\alpha(Q)$ contains $\sigma _x$.
\end{proof}

\subsection{Quadrics containing an elliptic normal curve}

One may perform the same construction starting with a degree $n+1$ normal elliptic curve $J\subset \PP^n$ and with its secant variety $\Sec J$, which has degree ${n\choose 2} -1$. The incidence variety $I_J\subset \Sec J\times \Sym^2J$ is defined as before and possesses two natural projections, that we still denote by $p$ and $q$.

It is useful to recall  that $\Sym^2(J)$ is the only elliptic ruled surface with invariant $e=-1$, cf., e.g. \cite{CaCi}. The natural $\PP^1$-bundle structure is given by the Albanese map $\pi: \Sym^2 J \to J$ sending  $x+y$ to $x\+ y$, where we let $\+$ (and $\ominus$) denote the group operation on $J$. Let $P_0$ be the neutral element. For each $P \in J$ the fiber 
\[ f_P:=\pi^{-1}(P)=\{ x+y \in \Sym^2(J) \; | \; x\+ y=P\}\]
is the $\PP^1$ defined by the linear system $|P+P_0|$. We denote the algebraic equivalence class of the fibers by $f$. 
For each $P \in J$, we define the curve $\sigma_P$
as the image of the section $J \to \Sym^2(J)$ mapping $Q$ to $P+ (Q \ominus P)$.
We let $\sigma$ denote the algebraic equivalence class of these sections, which are the ones with minimal self-intersection. We note that the diagonal $\Delta$ satisfies $\Delta \equiv 4\sigma-2f\equiv -2K_{\Sym^2(J)}$.
\begin{prop}\label{ellittica}
\begin{itemize}
\item[(i)] There exists an isomorphism 
$$
\alpha_J: |\I_{J/\PP^n}(2)|\longrightarrow |M|, \textrm{  with  }M\equiv (n-3)\sigma+2f,
$$ 
defined by setting $\alpha_J(Q):=q(p^{-1}(Q\cap\Sec J))$ for any $Q\in |\I_{J/\PP^n}(2)|$.
\item[(ii)] A quadric $Q\in |\I_{J/\PP^n}(2)|$ is singular at $x\in J$ if and only if $\alpha_J(Q)$ contains $\sigma _x$.
\end{itemize}
\end{prop}
\begin{proof}
For any $Q\in |\I_{J/\PP^n}(2)|$ one shows as in the proof of Proposition \ref{razionale} that $\alpha_J(Q)=q(p^{-1}(Q\cap\Sec J))$ is a curve $M\subset\Sym^2(J)$. By construction, the linear class of $M$ in $\Sym^2(J)$ does not depend on $Q$. In order to compute its numerical class, we first consider the degree $n$ cone $C(J)\subset \Sec J$ with vertex $P$ over the projection from $P$ of $J$; for degree reasons $C(J)$ intersects $Q$ along $J$ and $n-1$ bisecant lines, and thus $M\cdot \sigma_P=n-1$. Similarly, \color{black} $M\cdot f=n-3$ because the intersection of $Q$ with the rational normal scroll of degree $n-1$ in $\PP^n$ spanned by a $g^1_2$ on $J$ consists of the union of $J$ with $n-3$ bisecant lines.  Hence, $M\equiv (n-3)\sigma+2f$ and $h^0(M)=(n^2-n-2)/2=h^0(\I_{J/\PP^n}(2))$. The map $\alpha_J$ is thus an isomorphism. The proof of (ii) proceeds as the one of Proposition \ref{razionale}.
\end{proof}

\section{The case $\R_{6,2}$}

In this section we will prove the following theorem.

\begin{thm}\label{seidue}
The moduli spaces $\F_{23}^{\n,ns}$ and $\P_{6}$ are rational and the image of $r_{6,2}$ is a unirational divisor in $\R_{6,2}$.
\end{thm}

Take a general $(S,M,H)$ in $\F_{23}^{\n,ns}$. Then $g(R)=6$, $g(R')=5$ and $|R'|$ defines a morphism
$$\varphi_{R'}:S\longrightarrow \overline{S'}\subset \PP^5,$$
that contracts the curves $N_1,N_2$ to $2$ double points $x_1,x_2\in \overline{S'}$ and maps $N_3,\ldots,N_8$ to $6$ lines in  $\PP^5$. 

\begin{lemma} \label{lemma:Juve-15}
The divisor $
\Gamma:=R-N_3-\cdots-N_8
$
is represented by an irreducible curve satisfying
\[ \Gamma^2=-2, \;\; \Gamma\cdot R'=5 \;\; \mbox{and} \;\; \Gamma \cdot N_i =
\begin{cases}
  1, & i=1,2, \\
  2, & i=3,\ldots,8.
\end{cases}
\]
Furthermore, $h^0(R-N_1-\cdots-N_7)=0$. 
\end{lemma}

\begin{proof}
  The intersections are easily verified. Since $\Gamma \cdot R'=5$, the divisor $\Gamma$ is effective. 
As the irreducibility of $\Gamma$ is an open condition in $\F_{23}^{\n,ns}$, we may assume that \eqref{eq:pic9} holds.

 Let $\Gamma=D+E$ be an effective nontrivial decomposition. Write
  $D \sim aR+\sum_{i=1}^8 \frac{b_i}{2}N_i$ for some integers $a,b_i$ with $a \geq 0$ (as $D$ is effective). Then $E \sim (1-a)R-\frac{b_1}{2}N_1-\frac{b_2}{2}N_2 -\sum_{i=3}^8 \left(1+\frac{b_i}{2}\right)N_i$, which implies $a \leq 1$. Hence, $a \in \{0,1\}$ and, by symmetry, we may assume $a=0$, so that $D \sim \sum_{i=1}^8 \frac{b_i}{2}N_i$. It follows that all $b_i$ are nonnegative and even. We have thus proved that in any effective decomposition of $\Gamma$, all components but possibly one are supported on the $N_i$s. Thus we may write $\Gamma \sim \sum_{i=1}^8 \gamma_iN_i+\Gamma_0$, with $\gamma_i$ nonnegative integers and $\Gamma_0$ effective and irreducible. Then
  \[ -2 \leq \Gamma_0^2=\left(\Gamma -\sum_{i=1}^8 \gamma_iN_i\right)^2=
    -2-2\sum_{i=1}^8 \gamma_i^2-2(\gamma_1+\gamma_2+2\gamma_3+\cdots+2\gamma_8), \]
  implying $\gamma_i=0$ for all $i$.  The last statement is proved similarly. 
\end{proof}

Thus, $\varphi_{R'}$ maps $\Gamma$ to a rational normal curve $\Gamma\subset \PP^5$ that passes through the nodes $x_1,x_2\in \overline{S'}$ and has $N_3,\ldots,N_8$ as bisecant lines. 

Since $R'$ is not trigonal by \cite[Prop. 5.6]{DL}, the surface $\overline{S'}\subset \PP^5$ is the base locus of a net $\Sigma$ of quadrics. By projecting $\overline{S'}$ from the node $x_1$, we obtain a projective model $S^\dagger$ of $S$ in $\mathbb P^4$. Proceeding as in \cite[Prop. 5.6]{DL}, one can easily check that $R'-N_1$ is not hyperelliptic and thus $S^\dagger$ is the intersection of a cubic and a unique quadric $Q_1'$. The cone over $Q_1'$ with vertex in $x_1$ is a quadric $Q_1\subset \PP^5$ containing $\overline{S'}$, that is, $Q_1\in \Sigma$. Analogously, one proves the existence of a unique quadric $Q_2\in \Sigma$ singular at the point $x_2$. 
We may thus write
$$
\overline{S'}=Q_1\cap Q_2\cap Q_3,
$$
where $Q_3\in \Sigma$ is smooth. Since $Q_1,Q_2,Q_3$ contain $\Gamma$, Proposition \ref{razionale} provides three plane cubics $\alpha_\Gamma(Q_i)$  such that $\alpha_\Gamma(Q_1)=\sigma_{x_1}+\mathfrak{c}_1$ and $\alpha_\Gamma(Q_2)=\sigma_{x_2}+\mathfrak{c}_2$ for some conics $\mathfrak{c}_1,\mathfrak{c}_2\subset \PP^2$. The six bisecant lines $N_3,\ldots,N_8$ to $\Gamma$  are contained in $Q_1,Q_2,Q_3$ and thus define $6$ points $n_3,\ldots,n_8\in\alpha_\Gamma(Q_1)\cap \alpha_\Gamma(Q_2)\cap\alpha_\Gamma(Q_3)\subset\Sym^2(\Gamma)=\PP^2$ . Since for any $3\leq j\leq 8$ we have $N_j\cdot N_1=N_j\cdot N_2=0$  on $S$, neither $\sigma_{x_1}$ nor $\sigma_{x_2}$ passes through the points $n_j$. Therefore,  $n_3,\ldots,n_8\in  \mathfrak{c}_1\cap \mathfrak{c}_2$, and this implies $\mathfrak{c}_1\cap \alpha_{\Gamma}(Q_3)=\{ n_3,\ldots,n_8   \}$ and $\mathfrak{c}_2=\mathfrak{c}_1$, unless possibly when $\mathfrak{c}_1$ and $\mathfrak{c}_2$ are both reducible and share a common line $\l$ passing through $5$ of  the $n_i$ (without loss of generality, $n_3,\ldots,n_7$). However, if this were the case, there would exist a $2$-dimensional family of quadrics $Q_t$ in $\PP^5$ containing $\Gamma+N_3+\cdots+N_7$ and singular at $x_1, x_2$ (that correspond to the plane curves $\sigma_{x_1}+\l+\l_t$ for a varying line $l_t$). Any such $Q_t$ different from the uniquely determined $Q_1$ would cut out on $\overline{S'}$ a divisor whose strict transform in $S$ would lie in the linear system 
\begin{equation*}
|2R'-\Gamma- 2N_1-2N_2-N_3-\cdots-N_7|=|R-N_1-\cdots-N_7|,
\end{equation*}
which is empty  by Lemma \ref{lemma:Juve-15}. We thus conclude that $\mathfrak{c}_1$ is irreducible and $\mathfrak{c}_2=\mathfrak{c}_1$.

Vice versa, using the notation of \S \ref{federico} we prove the following: 

\begin{prop}\label{ita}
Let $Q_1,Q_2,Q_3\subset \PP^5$ be quadrics containing a fixed rational normal curve $\Gamma\subset\PP^5$. Assume that $\alpha_\Gamma(Q_1)=\sigma_{x_1}+\mathfrak{c}$, $\alpha_\Gamma(Q_2)=\sigma_{x_2}+\mathfrak{c}$, $\alpha_\Gamma(Q_3)=\mathfrak{e}$, where $\mathfrak{c}\in |\mathcal O_{\PP^2}(2)|$ and $\mathfrak{e}\in |\mathcal O_{\PP^2}(3)|$ intersect at $6$ distinct general points, and $x_1\neq x_2$. Then for general such $Q_1,Q_2,Q_3$ the following hold:
\begin{itemize}
\item[(i)] The minimal desingularization $S$ of $\overline{S'}:= Q_1\cap Q_2\cap Q_3\subset \PP^5$ is a non-standard Nikulin surface of genus $23$.
\item[(ii)] By varying $x_1,x_2\in \Gamma$, $\mathfrak{c}\in |\mathcal O_{\PP^2}(2)|$ and $\mathfrak{e}\in |\mathcal O_{\PP^2}(3)|$, one obtains all Nikulin surfaces in a dense open subset of $\F_{23}^{\n,ns}$.
\end{itemize}
\end{prop}

\begin{proof} 
In the case where the points $x_1,x_2\in \Gamma$ and $n_3,\ldots,n_8 \in \PP^2$ arise starting from a Nikulin surface, there exists a unique quadric in $\PP^5$ singular at $x_1$ and containing $\Gamma, N_3,\ldots,N_8$ because the plane conic $\mathfrak{c}$ through $n_3,\ldots,n_8$ is unique; the same holds for $x_2$. This ensures that the projective model of a general Nikulin surface as above is limit of complete intersections as in (i) where $x_1,x_2\in \Gamma$ and $n_3,\ldots,n_8 \in \PP^2$ are chosen generically. 

A surface $\overline{S'}$ as in (i) has trivial canonical bundle and two double points at $x_1$ and $x_2$. By generality (as this holds true when $\overline{S'}$ is the projective model of a Nikulin surface), $\overline{S'}$ has no further singularity. Denoting by $\nu:S\to \overline{S'}$ the minimal desingularization, the surface $S$ is a $K3$ surface and the curves $N_1=\nu^{-1}(x_1)$ and $N_2=\nu^{-1}(x_2)$ are ($-2$)-curves. Since $Q_1,Q_2,Q_3$ all contain the $6$ bisecant lines to $\Gamma$ corrisponding to the $6$ intersection points $\mathfrak{c}\cap \mathfrak{e}$, the surface $\overline{S'}$ contains $6$ bisecant lines to $\Gamma$, whose inverse images on $S$ we call $N_3,\ldots,N_8$. By construction, the eight ($-2$)-curves $N_i$ on $S$ are all disjoint. We set $R':=\nu^*\mathcal{O}_{\overline{S'}}(1)$, and observe that the divisor
$2R'-\sum_{i=1}^8N_i-2\Gamma$
has self intersection $0$ and degree $0$ with respect to $R'$, and is thus trivial. As a consequence, we obtain that
$$
\sum_{i=1}^8N_i=2R'-2\Gamma
$$
is $2$-divisible in $\Pic(S)$. Setting $M:=(N_1+\cdots + N_8)/2$, $R:=\nu^*\mathcal{O}_{\overline{S'}}(\Gamma+N_3+\cdots+N_8)$ and $H:=R+R'+M$, one verifies that $(S,M,H)\in \F_{23}^{\n,ns}$, thus obtaining (i).

We perform a parameter count to prove (ii). The choice of $\mathfrak{c}\in |\mathcal O_{\PP^2}(2)|$ and of $\sigma_{x_1},\sigma_{x_2}$  depends on $5+2=7$ parameters. For any such triple $\mathfrak{c}, \sigma_{x_1},\sigma_{x_2}$, let $Z(\mathfrak{c}, x_1+x_2)$ be the Schubert variety of all nets in $|\mathcal O_{\PP^2}(3)|=\PP^9$ containing the $1$-dimensional space spanned by $\sigma_{x_1}+\mathfrak{c}$ and $\sigma_{x_2}+\mathfrak{c}$.  Then $Z(\mathfrak{c}, x_1+x_2)$ is a projective space of dimension $7$ and we arrive at $7+7=14$ parameters. Two nets $\Sigma,\Sigma'$ of quadrics containing $\Gamma$ define the same surface $\overline{S'}$ up to projectivities as soon as they differ by an automorphism of $\PP^5$ fixing $\Gamma$, that is, by an automorphism of $\Gamma$ itself. Since $\mathrm{Aut}(\Gamma)\simeq \PGL(2)$, this brings the number of parameters down to $11$, as wanted.
\end{proof}

\begin{proof}[Proof of Theorem \ref{seidue}]
Let $U\subset |\mathcal{O}_{\PP^2}(2)|\times \Sym^2(\Gamma)$ be the open subset parametrizing pairs $(\mathfrak{c}, x_1+x_2)$ such that  $\mathfrak{c}$ intersects transversally the diagonal $\Delta$,  $\sigma_{x_1}$ and $\sigma_{x_2}$. Set $$\mathcal Z:=\left\{(\Sigma,\mathfrak{c}, x_1+x_2)\,\,|\,\,  (\mathfrak{c}, x_1+x_2)\in U,\,\, \Sigma\in Z(\mathfrak{c}, x_1+x_2)   \right\},$$ 
where $Z(\mathfrak{c}, x_1+x_2)$ is the Schubert variety defined in the previous proof. The forgetful map $\pi:\mathcal Z\to U$ gives $\mathcal Z$ the structure of a $\PP^7$-bundle over $U$. Let $G$ be the subgroup of $\PGL(3)$ fixing the plane conic $\Delta\subset \PP^2$. Then $G$ is naturally isomorphic to $\PGL(2)$ and acts linearly on $\mathcal Z$ and on $U$ in  the natural way: for every $f\in G$, we have $f(\Sigma,\mathfrak{c}, x_1+x_2)=(f_*(\Sigma),f(\mathfrak{c}), f(x_1)+f(x_2))$ where $f_*:|\mathcal{O}_{\PP^2}(3)|\to |\mathcal{O}_{\PP^2}(3)|$ is the induced map. The moduli space $\F_{23}^{\n,ns}$ is then birational to the quotient $\mathcal Z / G$. We want to show that $\mathcal Z / G$ is a $\PP^7$-bundle over $U/G$. By
Kempf's descent lemma \cite{DN}, it is enough to check that for every $o=(\mathfrak{c}, x_1+x_2)\in U$ the stabilizer of $\pi^{-1}(o)$ in $G$ is trivial: this is true because the conic $\mathfrak{c}$ intersects $\Delta$ in $4$ distinct points and  there are no non-trivial automorphisms of $\PP^1$ mapping a set of $4$ general points to itself. The rationality of $\F_{23}^{\n,ns}$ thus follows as soon as we show that  $U/G$ is rational. The quotient $(|\mathcal{O}_{\PP^2}(2)|\times \Sym^2(\Gamma))/G$ is birational to a $\PP^1$-bundle over the quotient $\left(\Sym^4(\Delta)\times \Sym^2(\Delta)\right)/\Aut(\Delta)$, where $\Delta\simeq \PP^1$. For a general $(w_1+w_2+w_3+w_4,w_5+w_6)\in\Sym^4(\PP^1)\times \Sym^2(\PP^1)$ the double cover of $\PP^1$ ramified at the $6$ points $w_1,\ldots,w_6$ is a genus $2$ curve $D$. Denoting by $y_i\in D$ the inverse image of $w_i$, it is straightforward to check that  $\mathcal O_D(y_5- y_6)\in \Pic^0(D)[2]$, that is, the pair $(D,\mathcal O_D(y_5-y_6))$ lies in $\mathcal R_2$. Two elements $(w_1+w_2+w_3+w_4,w_5+w_6)$ and $(w_1'+w_2'+w_3'+w_4',w_5'+w_6')$ define the same Prym curve in $\mathcal R_2$ if and only if they are obtained one from the other acting with an automorphism of $\PP^1$. By a parameter count, a general point of $\mathcal R_2$ is obtained with this construction and thus the quotient $\left(\Sym^4(\Delta)\times \Sym^2(\Delta)\right)/\Aut(\Delta)$ is birational to $\mathcal R_2$, which is rational \cite{Do}. This implies the rationality of $U/G$ and thus of the moduli spaces $\F_{23}^{\n,ns}$ and $\P_{6}$.\color{black}

In order to prove that the image of map $r_{6,2}:\P_{6}\to \mathcal R_{6,2}$ is a unirational divisor, it is enough to show that a general fiber of $r_{6,2}$ is $1$-dimensional. Take a general $(S,M,H,C)\in P_{6}$ and let $(C,x_1+x_2,M^\vee|_C)$ be its image in $\mathcal R_{6,2}$.

We claim that the space $|\I_{C/\PP^5}(2)|$ of quadrics containing $C\subset\PP^5=\PP(H^0(\omega_C\otimes M)^\vee)$ is three-dimensional. Indeed, this follows from the short exact sequence
\[
  \xymatrix{
    0 \ar[r] & \I_{C/\PP^5}(2) \ar[r] & \O_{\PP^5}(2) \ar[r] & \O_C(2) \ar[r] & 0}\]
once we show that the restriction map $r:H^0( \O_{\PP^5}(2)) \to H^0(\O_C(2))$ is surjective.
To prove the latter, we note that $r$ is the composition of the restriction map \linebreak  $H^0( \O_{\PP^5}(2)) \to H^0(\O_{\overline{S'}}(2))$, which is surjective as $\overline{S'}$ is a complete intersection, and the restriction map $H^0(\O_{\overline{S'}}(2))\to H^0(\O_C(2))$. The latter equals the restriction map $H^0(\O_S(2R')) \to H^0(\O_C(2))$, which has cokernel
\[ H^1(\O_S(2R'-R))=H^1(\O_S(\Gamma+N_1+N_2))=H^1(\O_S(-(\Gamma+N_1+N_2))^{\vee};\] this is zero since $\Gamma+N_1+N_2$ is connected. Thus, the claim is proved.

Since the linear system  $|\I_{C/\PP^5}(2)|$ cuts out a curve on $\overline{S'}$ whose inverse image on $S$ lies in $|2R'-R|=|\Gamma+N_1+N_2|$, which has only one member, we see that the base locus of $|\I_{C/\PP^5}(2)|$ is $C\cup \Gamma$. Since $\overline{S'}=Q_1\cap Q_2\cap Q_3$ with $Q_1$ and $Q_2$ singular at the double points $x_1$ and $x_2$ respectively, the base locus of a general net of quadrics in $|\I_{C/\PP^5}(2)|$ containing the $1$-dimensional space generated by $Q_1$ and $Q_2$ is the $2$-nodal model  in $\PP^5$ of a Nikulin surface in $\F_{23}^{\n,ns}$. We have a $\PP^1$ of such nets and they are not projectively equivalent because there is no projectivity of $\PP^5$ fixing $C$ and $\Gamma$. This proves that the fiber of $r_{6,2}$ over $(C,x_1+x_2,M^\vee|_C)$ is $1$-dimensional.
\end{proof}

\section{The case $\R_{5,4}$}
In this section we will prove the following theorem.

\begin{thm}\label{54}
The moduli spaces $\widehat{\F}_{21}^{\n,ns}$, $\F_{21}^{\n,ns}$  and $\widehat\P_{5}$ are unirational and $r_{5,4}$ is dominant. In particular, the moduli space $\R_{5,4}$ is unirational, too.
\end{thm}

Take a general $(S,M,H,R)$ in $\widehat{\F}_{21}^{\n,ns}$.
We have $g(R)=g(R')=5$, and $|R'|$ defines a morphism
$$\varphi_{R'}:S\longrightarrow \overline{S'}\subset \PP^5,$$
that contracts the curves $N_1,\ldots,N_4$ to $4$ double points $x_1,\ldots,x_4\in \overline{S'}$ and maps $N_5,\ldots,N_8$ to $4$ lines in  $\PP^5$.
Arguing as in the proof of Lemma \ref{lemma:Juve-15}, one proves the following:

\begin{lemma} \label{lemma:Juve-15-2}
The divisor
$J:=R-N_5-\cdots-N_8$
is effective, with
\[ J^2=0, \;\; J\cdot R'=6\;\; \mbox{and} \;\; J \cdot N_i =
\begin{cases}
  1, & i=1,\ldots,4 \\
  2, & i=5,\ldots,8,
\end{cases}
\]
and all members of $|J|$ are irreducible. 
\end{lemma}

It follows that $|J|$ is a base point free elliptic pencil and that
$\varphi_{R'}$ maps any fixed smooth $J$ in its linear system 
to an elliptic normal curve $J\subset \PP^5$ passing through the nodes $x_1,\ldots x_4\in \overline{S'}$ and having $N_5,\ldots,N_8$ as bisecant lines.  As in the previous section, one shows that $\overline{S'}$ is the base locus of a net  $\Sigma$ of quadrics containing $J$ and for every $1\leq i\leq 4$ the net $\Sigma$ contains a unique quadric $Q_i$ singular at $x_i$.  We show that there is no further quadric in $\PP^5$ that is singular at $x_i$ and contains $\Gamma,N_5,\ldots,N_8$. Indeed, such a quadric would cut out a divisor on $\overline{S'}$  whose strict transform on $S$ would lie in
\begin{equation}\label{bello} |2R'-J-2N_i-N_5-\cdots -N_8|=|J+(N_1+N_2+N_3+N_4-N_i)-N_i |.
\end{equation}
However, this linear system is empty. Indeed, since its intersection with $N_j$ is $-1$ for $j \in \{1,2,3,4\} \setminus \{i\}$, if \eqref{bello} were nonempty then
$J-N_i$ would be linearly equivalent to a nontrivial effective divisor, contradicting Lemma \ref{lemma:Juve-15-2}.

\begin{prop}\label{preparatoria}
 Fix an elliptic normal curve $J\subset\PP^5$. Take $3$ general points $x_1,x_2,x_3\in J$ and $4$ general lines $N_5,\ldots, N_8$ bisecant to $J$. Then the following hold:
\begin{itemize}
\item[(i)] For each $1\leq i\leq 3$, there exists a unique quadric $Q_i\subset\PP^5$ that is singular at $x_i$ and contains $N_5,\ldots, N_8, J$.
\item[(ii)] For $Q_1,Q_2,Q_3$ as in (i), the minimal desingularization $S$ of the complete intersection $\overline{S'}:= Q_1\cap Q_2\cap Q_3\subset \PP^5$ is a marked non-standard Nikulin surface of genus $21$.
\item[(iii)] By varying $x_1,x_2,x_3\in J$ and $n_5,\ldots,n_8\in \Sym^2J$, one obtains all members in a dense open subset of $\widehat{\F}_{21}^{\n,ns}$. 
\end{itemize}
\end{prop}

\begin{proof} 
By the above discussion, (i) holds when $x_1,x_2,x_3$ and $N_5,\ldots,N_8$ are nodes and lines on the projective model of a general Nikulin surface, which is thus limit of general complete intersections as in (ii). By generality, (i) holds for general choices of $x_1,x_2,x_3\in J$ and $n_5,\ldots,n_8\in \mathrm{Sym}^2J$ as well. 

A surface $\overline{S'}$ as in (ii) has trivial canonical bundle, contains the lines $N_5,\ldots, N_8$ and has $3$ double points at  $x_1,x_2,x_3$.  By generality, $\overline{S'}$ has at most $4$ double points as this occurs when one starts with a Nikulin surface. We define $N_1,N_2,N_3$ as the exceptional divisors of the minimal desingularization $\nu:S\to \overline{S'}$ over the point $x_1,x_2,x_3$, respectively. We now show that $\overline{S'}$ has automatically a fourth double point. We set $R':=\nu^*\mathcal{O_{\overline{S'}}}(1)$ and consider the following divisor on $S$: 
$$
N_4:=2R'-2J-N_1-N_2-N_3-N_5-N_6-N_7-N_8.
$$
It is straightforward to check that $N_4^2=-2$ and $N_4 \cdot J=1$, so that $N_4$ is effective by Riemann-Roch and Serre duality. Moreover, 
$N_4\cdot R'=0$, whence $N_4$ is contracted to a fourth singular point $x_4$ of $\overline{S'}$, which is a double point if $N_4$ is irreducible. By construction $N_i\cdot N_j=0$ for $i\neq j$  and $\sum N_i \sim 2M$, with $M:=R'-J$. Setting $R:=\nu^*\mathcal{O}_{\overline{S'}}(J+N_5+\cdots+N_8)$ and $H:=R+R'+M$, it is easy to check that, if $N_4$ is irreducible, then  $(S,M,H,R)\in \widehat{\F}_{21}^{\n,ns}$. If $N_4$ were not irreducible, then the Picard group of $S$ would have rank $>9$; whence, by standard theory of $K3$ surfaces, the obtained family of $K3$ surfaces would have dimension $<20-9=11$. Thus, (i) will follow once we show that this construction provides an $11$-dimensional family of $K3$ surfaces.

By (i),  our construction depends on $3+8=11$ parameters, corresponding to the choice of $x_1,x_2,x_3\in J$ and $n_5,\ldots,n_8\in\Sym^2J$. Since there are no projectivities of $\PP^5$ fixing an elliptic normal curve and the automorphism group of any $K3$ surface is discrete, we obtain an $11$-dimensional family of $K3$ surfaces, thus proving (ii) and (iii). 
\end{proof}

\begin{proof}[Proof of Theorem \ref{54}]
  By Proposition \ref{preparatoria}, a general  $(S,M,H,R)\in
\widehat{\F}_{21}^{\n,ns}$ is obtained as the minimal desingularization of the intersection $\overline{S'}$ of $3$ quadrics $Q_1,Q_2,Q_3$ in $\PP^5$ containing a fixed elliptic normal curve $J\subset \PP^5$ and $4$ bisecant lines $N_5,\ldots, N_8$ to it and such that each $Q_i$ is singular at one point $x_i\in J$; note that the line bundle $R'=H-M-R$ is the one giving the map to $\overline{S'}\subset\PP^5$. 

To obtain the unirationality of $\widehat{\F}_{21}^{\n,ns}$, $\F_{21}^{\n,ns}$, $\widehat\P_{5}$, we recall that the strict transform of $J$ defines an elliptic pencil $|J|$ on $S$, that contains singular fibers. By Proposition \ref{lemma:Juve-15-2} all members of $|J|$ are irreducible. Thus $|J|$ contains an irreducible rational curve $J_0$, which is mapped isomorphically by $\varphi_{R'}$ to a curve passing simply through the $4$ double points $x_1,\ldots,x_4\in\overline{S'}$ and having $N_5,\ldots, N_8$ as bisecant lines.

Since $J_0\in |J|$ and \eqref{bello} is not effective, the quadrics $Q_1,Q_2,Q_3$ (and thus $\overline{S'}$) can be recovered from the datum of $x_1+x_2+x_3\in \Sym^3(J_0)$ and $n_5,\ldots, n_8\in\Sym^2(J_0)$. By upper semicontinuity, having fixed the embedding $J_0\subset \PP^5$, there is only one quadric in $\PP^5$ singular at a general point of $J_0$ and containing $J_0$ and $4$ general bisecant lines to $J_0$. Furthermore, starting from $3$ general points on $J_0$ and $4$ general bisecant lines to $J_0$, one obtains as in the proof of Proposition \ref{preparatoria} a
non-standard Nikulin surface by desingularizing the intersection of the uniquely determined $3$ quadrics. In other words, there is a finite  (corresponding to the choice of $3$ out of $4$ lines and of $J_0\in |J|$) dominant map $\Sym^3{J_0}\times \Sym^4(\Sym^2J_0)\to \widehat{\F}_{21}^{\n,ns}$ (as there was for $J$ by Proposition \ref{preparatoria}); since the domain is rational, this implies the unirationality of  $\widehat{\F}_{21}^{\n,ns}$, $\F_{21}^{\n,ns}$ and  $\widehat\P_{5}$.   

We now show that $r_{5,4}:\widehat\P_{5}\to \mathcal R_{5,4}$ is birational, thus obtaining the unirationality of $\mathcal R_{5,4}$. Take a general $(S,M,H,R,C)\in \widehat\P_{5}$ and let $(C,x_1+x_2+x_3+x_4,M^\vee|_C)$ be its image in $\mathcal R_{5,4}$. We need to show that $\overline{S'}$ is the only surface that contains $C\subset \PP(H^0(\omega_C\otimes M)^\vee)=\mathbb P^5$ and $4$ lines and can be written as the base locus of a net of quadrics supporting a quadric singular at $x_i\in C$ for every $1\leq i\leq 4$. It is then enough to specialize $C$ to a curve $J+N_5+\cdots+N_8\in |R|$ and apply Proposition \ref{preparatoria}(ii).\color{black}
\end{proof}

\section{The case $\R_{5,2}$}

In this section we will prove the following theorem.

\begin{thm}\label{thm:5,2}
  The moduli spaces $\F_{19}^{\n,ns}$ and $\P_{5}$ are unirational.  The map $r_{5,2}$ is dominant; in particular the moduli space $\R_{5,2}$ is unirational.
\end{thm}

Take a general $(S,M,H)$ in $\F_{19}^{\n,ns}$, so that $g(R)=5$ and $g(R')=4$. The morphism
$$\varphi_{R'}:S\longrightarrow \overline{S'}\subset \PP^4,$$
contracts the curves $N_1,N_2$ to $2$ double points $x_1,x_2 \in \overline{S'}$ and maps $N_3,\ldots,N_8$ to $6$ lines in  $\PP^4$.

\begin{lemma} \label{lemma:Juve-15-4}
The divisor
$
\Gamma:=R-N_3-\cdots-N_7
$
on $S$ is represented by an irreducible curve satisfying
\[ \Gamma^2=-2, \;\; \Gamma\cdot R'=4 \;\; \mbox{and} \;\; \Gamma \cdot N_i =
\begin{cases}
  1, & i=1,2 \\
  2, & i=3,4,5,6,7,\\
  0, & 8. 
\end{cases}
\]
Moreover, $h^0(3R'-2(N_1+N_2))=21$, $h^0(S,\O_S(3R'-2(N_1+N_2)-\Gamma))=12$ and
$h^0(S,\O_S((3R'-2(N_1+N_2)-\Gamma-(N_3+\cdots+N_7))=2$ 
\end{lemma}

\begin{proof}
 The proof of the first statement is similar to the proof of Lemma \ref{lemma:Juve-15-2}. In the same way one proves that the divisor
  $\Gamma':=R'-N_1-N_2-N_8$ is represented by an irreducible $(-2)$-curve satisfying $\Gamma' \cdot \Gamma=2$, and that $\Gamma+\Gamma' \sim 3R'-2(N_1+N_2)-\Gamma-(N_3+\cdots+N_7)$. One can use this to show that the divisors in the last statement are all big and nef, and thus compute their cohomology.
\end{proof}

As a consequence, $\Gamma$ is mapped by $\varphi_{R'}$ to a rational normal quartic $\Gamma\subset \PP^4$ passing through the nodes $x_1,x_2\in \overline{S'}$. The $5$ lines $N_3,N_4,N_5,N_6,N_7$ are bisecant to $\Gamma$, whereas the line $N_8$ is disjoint from $\Gamma$.

As $R'$ is not hyperelliptic by \cite[Prop. 5.6]{DL}, the surface  $\overline{S'}\subset \PP^4$ is the complete intersection of a quadric and a cubic.

\begin{lemma} \label{lemma:incubichesing2}
  The surface $\overline{S'}$ is contained in precisely a web of cubics singular at $x_1,x_2$ and in a unique quadric.
\end{lemma}

\begin{proof}
 The projective normality of $\overline{S'}\subset \mathbb P^4$ directly implies the uniqueness of the quadric $Q$ containing it. Since $\overline{S'}$ has only two nodes at  $x_1,x_2$, the quadric $Q$ is necessarily smooth if we prove the existence of a cubic containing $\overline{S'}$ and singular at  $x_1,x_2$.

Setting $\mathfrak{n}: = x_1 + x_2 $, the linear system $|\I^2_{\mathfrak{n}/\PP^4}(3))|$ of cubics singular at $\mathfrak n$ has dimension at least  $24$, while the linear system $|3R'-2(N_1+N_2)|$ on $S$  has dimension $20$ by Lemma \ref{lemma:Juve-15-4}; therefore, the linear system $\Sigma$ of cubics singular at $\mathfrak n$ and containing $\overline{S'}$ is at least a web.  The quadric $Q$ containing $\overline{S'}$ is thus smooth and the restriction of $\Sigma$ to $Q$ is contained in $\mathbb P(H^0(\I_{\overline{S'}/Q}(3))) \cong \mathbb P(H^0(\O_Q))$, which is a point. The restriction map $r_Q:\Sigma \dashrightarrow \Sigma|_Q$ is the projection from the sublinear system of $\Sigma$ consisting of cubics that contain $Q$. Any such cubic splits as $Q\cup H$, where $H$ is a hyperplane such that $Q \cup H$ is singular at $x_1,x_2$; as $Q$ is smooth, this forces $H$ to pass through $x_1,x_2$ so that $H$ moves in a net. We conclude that $\dim\Sigma=3$.
\end{proof}

The following result will be used to prove the unirationality of $\F_{19}^{\n,ns}$ and the dominance of $r_{5,2}$.

\begin{prop}\label{preparatoria5,2}
Let $\Gamma\subset\PP^4$ be a fixed rational normal curve. Take $2$ general points $x_1,x_2 \in \Gamma$ and $5$ general secant lines $N_3,N_4,N_5,N_6,N_7$ to $\Gamma$. Then the following hold:
\begin{itemize}
\item[(i)] There exists a unique quadric  $Q \subset\mathbb P^4$ and a $5$-dimensional linear system of cubics $Y\subset\mathbb P^4$ such that both $Q$ and $Y$ contain $\Gamma+N_3+\cdots+N_7$ and $Y$ is singular at $x_1,x_2$.
\item[(ii)] For general $Q$ and $Y$ as in (i), the minimal desingularization $S$ of the complete intersection $\overline{S'}:= Q\cap Y\subset \PP^4$ is a non-standard Nikulin surface of genus $19$.
\item[(iii)] By varying $x_1,x_2 \in \Gamma$ and $n_3,\ldots,n_7\in \Sym^2\Gamma$, one obtains all Nikulin surfaces in a dense open subset of $\F_{19}^{\n,ns}$.
\end{itemize}
\end{prop}

\begin{proof}
  The existence of a unique quadric $Q$ follows from Proposition \ref{razionale}.

Setting $\mathfrak{n}: = x_1 + x_2 $, we have $h^0(\mathcal I^2_{\mathfrak{n}/\PP^4}(3))\geq 25$ and thus the kernel of the
  restriction map
  \begin{equation} \label{eq:rho4,4}
    \rho:  H^0(\mathcal I^2_{\mathfrak{n}/\PP^4}(3)) \longrightarrow
    H^0(\mathcal \O_{\Gamma}(3)(-2(x_1+x_2))\cong H^0(\O_{\PP^1}(8)) \cong \CC^9
    \end{equation}
has dimension $\geq16$. The space of cubics $Y$ as in the statement is the projectivization of the kernel of
\begin{equation} \label{eq:restr2}
  r: \Ker \rho \longrightarrow H^0(\oplus_{i=3}^7 \O_{N_i}(3-2))
\cong H^0(\O_{\PP^1}(1))^{\+ 5} \cong \CC^{10},
\end{equation}
and thus has projective dimension $\geq 5$. By semicontinuity, to prove (i) it is enough to show that equality holds when $x_1,x_2$ and  $N_3,N_4,N_5,N_6,N_7$ are the nodes and lines of the projective model $\overline{S'}$ of a Nikulin surface as above. In this case the map $r$ factors as:
  \begin{equation}\label{fact} r:  \Ker \rho \stackrel{r_{\overline{S'}}}{\longrightarrow}
    \left(\Ker \rho\right)|_{\overline{S'}}  \stackrel{r'}{\longrightarrow} H^0(\O_{\PP^1}(1))^{\+ 5} \cong \CC^{10},\end{equation}
    with $\overline{S'}$ as in (ii).
Since $\dim (\Ker \rho) \geq 16$, $\dim \Ker r_{\overline{S'}}=4$ by Lemma \ref{lemma:incubichesing2} and $\dim \left(\Ker \rho\right)|_{\overline{S'}} \leq h^0(S,\O_S(3R'-2(N_1+N_2)-\Gamma)=12$ by Lemma \ref{lemma:Juve-15-4},
 we must have 
$\dim (\Ker \rho) =16$ and 
\[ \left(\Ker \rho\right)|_{\overline{S'}}  \cong H^0(S,\O_S(3R'-2(N_1+N_2)-\Gamma) \cong \CC^{12}.\]
In particular, $r'$ can be identified with the restriction map
  \[ \CC^{12} \cong H^0(S,\O_S(3R'-2(N_1+N_2)-\Gamma) \longrightarrow  H^0(\O_{N_3 \cup \dots \cup N_7}(1)) \cong \CC^{10}\]
  on $S$, so that
  \[ \Ker r' \cong H^0(\O_S(3R'-2(N_1+N_2)-\Gamma-(N_3+\cdots+N_7)))=\CC^2,\]
  again by Lemma \ref{lemma:Juve-15-4}. This yields the surjectivity of both $r'$ and $r$. Consequently, $\dim \Ker r =6$, which finishes the proof of (i).

We stress that, having proved (i) also in the case where the point and the lines are nodes and lines (possibly in special position) on the projective model of a Nikulin surface, the family of surfaces as in (ii) obtained by varying $x_1,x_2 \in \Gamma$ and $n_3,\ldots,n_7\in \Sym^2\Gamma$ always contains the projective model of a Nikulin surface.
Thus a surface $\overline{S'}:= Q\cap Y\subset \PP^4$ as in (ii) has trivial canonical bundle and $2$ ordinary double points at $x_1,x_2$, while by generality it is smooth elsewhere. Let $\nu:S\to \overline{S'}$ be its minimal desingularization and $N_1,N_2$ be the exceptional divisors. We set $R':=\nu^*\mathcal O_{\overline{S'}}(1)$ and define $N_8:=2\Gamma-2R'+N_1+\cdots+N_7$. One easily verifies that $N_8^2=-2$ and
         $N_8 \cdot N_i=0$ for $i\leq 7$. Moreover, setting $M:=R'+N_8-\Gamma$, one has 
         $N_1+\cdots+N_8 \sim 2M$. Defining $R:=\Gamma+N_3+\cdots+N_7$ and $H:=R'+R+M$, it comes straightforward that $(S,M,H)\in \F_{19}^{\n,ns}$, as soon as $N_8$ is irreducible. As in the proof of of Proposition \ref{preparatoria}, (ii) will
follow as soon as we show that the constructed $K3$ surfaces move in a family of dimension $\geq 11$. Our construction depends on 
$$
\dim\Sym^2(\Gamma)+\dim \Sym^5(\Sym^2\Gamma)+5-3+\dim \Aut(\Gamma)=11
$$
parameters, where the term $+5$ follows from (i) and the term $-3$ from Lemma \ref{lemma:incubichesing2}. This concludes the proof of both (ii) and (iii).
\end{proof}

\begin{proof}[Proof of Theorem \ref{thm:5,2}]
We proceed as in the proof of Theorem \ref{seidue}, fixing a plane conic $\Delta$ in order to identify $\Sym^2(\Gamma)\cong\PP^2$. We thus consider $\Hilb^5({\PP^2})\times \Hilb^2(\PP^1)$ as a parameter space for pairs $(\mathfrak l,\mathfrak n)$, where $\mathfrak l$ is a set of five unordered  bisecant lines to $\Gamma$  and $\mathfrak{n}$ is a set of two unordered points on $\Gamma$. Over an open set of it  we have a $\PP^5$-bundle $\MM$ whose fibre over a point $\left(\mathfrak l,\mathfrak{n} \right)$ is the $\PP^5$  of cubics $Y$ containing
$\Gamma \cup \mathfrak l$ and singular at $\mathfrak{n}$ (by Proposition \ref{preparatoria5,2}(i)).
The group $\PGL(2)\subset \PGL(3)$ of projectivities of $\PP^2$ fixing $\Delta$ acts an all the spaces $\Hilb^5({\PP^2})$, $\Hilb^2(\PP^1)\cong \PP^2$ and $\MM$. Furthermore, denoting by $U\subset \Hilb^5({\PP^2})$  the open subset parametrizing $5$-tuple of points in general position, the stabilizer of any point $\mathfrak n\in U$ in $\PGL(2)$ is trivial. Therefore, Kempf's descent lemma implies that $\PP_U:=U\times \Hilb^2({\PP^1})\to U$ descends to a $\PP^2$-bundle $\PP_U\sslash\PGL(2)$ on $U/\PGL(2)$ and, similarly, the restriction $\MM_U$ of $\MM$ to $\PP_U$  descends to a $\PP^5$-bundle $\MM_U\sslash\PGL(2)\to \PP_U\sslash\PGL(2)$. Hence, $\MM_U$ is rational as soon as $U/\PGL(2)$ is. This is proved by looking at the natural projection $U/\PGL(2) \to U/ \PGL(3)$ and observing that $U/ \PGL(3)$ is a rational surface by Castelnuovo's Theorem and
 $$
 \PGL(3)/\PGL(2)\simeq |\O_{\PP^2}(2)|\simeq \PP^5.$$ 
We have thus proved that $\MM_U$ is rational of dimension
$2+5+2+5=14$. 
By construction, it admits a natural dominant map $\tau: \MM_U\to \tilde \F$, where $\tilde \F$ is a $6:1$ cover of $\F_{19}^{\n,ns}$ corresponding to the selection of one out of $N_3,\ldots,N_8$ to define the class of $\Gamma$. A general fiber of $\tau$ is isomorphic to $\PP^3$ by Lemma \ref{lemma:incubichesing2}. Hence, $\widetilde \F$ is stably rational and $\F_{19}^{\n,ns}$, $\P_{5}$  and $\P'_4$ \color{black} are unirational. The unirationality of $\R_{5,2}$ follows once we show that the map $r_{5,2}:\P_{5} \to \R_{5,2}$ is dominant, or equivalently, that a general fiber of it has dimension $2$. Take a general $(S,M,H,C)\in \P_5$. By Lemma \ref{lemma:incubichesing2} the projective model $\overline{S'}\subset \PP(H^0(S,R)^\vee)=\PP^4$ of $S$ contains $6$ lines and is the complete intersection of a uniquely determined quadric and a cubic (moving in a web) singular at $x_1,x_2\in C$. Therefore, it is enough to show that $C\subset \mathbb P(H^0(C,\omega_C\otimes M)^\vee)=\PP^4$ is contained in a unique quadric and in a $5$-dimensional linear system of cubics singular at $x_1,x_2$ such that the intersection $Y\cap Q$ contains $6$ lines. To show this we specialize $C$ to the curve $\Gamma+N_3+\cdots +N_7\in |R|$ and apply Proposition \ref{preparatoria5,2}(i).
\end{proof}

\section{The case  $\R_{4,4}$}

The aim of this section is to  prove the following theorem.

\begin{thm}\label{thm:4,4}
The moduli spaces $\widehat{\F}_{17}^{\n,ns}$, $\F_{17}^{\n,ns}$ and $\widehat\P_{4}$ are unirational and $r_{4,4}$ is dominant. In particular, the moduli space $\R_{4,4}$ is unirational, too.
\end{thm}

Take a general $(S,M,H,R)$ in $\widehat{\F}_{17}^{\n,ns}$, so that $g(R)=g(R')=4$. The morphism
$$\varphi_{R'}:S \longrightarrow \overline{S'}\subset \PP^4,$$
contracts the curves $N_1,\ldots,N_4$ to $4$ double points $x_1,\ldots,x_4\in \overline{S'}$ and maps $N_5,\ldots,N_8$ to $4$ lines in  $\PP^4$.

\begin{lemma} \label{lemma:Juve-15-3}
The divisor
$
\Gamma:=R-N_5-\cdots-N_8
$
on $S$ is represented by an irreducible curve satisfying
\[ \Gamma^2=-2, \;\; \Gamma\cdot R'=4 \;\; \mbox{and} \;\; \Gamma \cdot N_i =
\begin{cases}
  1, & i=1,2,3,4, \\
  2, & i=5,6,7,8.
\end{cases}
\]

Moreover, we have
\begin{eqnarray}
\label{eq:gamma2}
  h^0(\O_S(R'-N_1-\cdots-N_4))=1, \\
  \label{eq:quadriche} h^0(\O_S(2R'-\Gamma-(N_5+\cdots+N_8))=1, \;\;  \\ 
  \label{eq:cubic0} h^0(\O_S(3R'-2(N_1+\cdots+N_4))=13, \;\;  \\
  \label{eq:cubic} h^0(\O_S(3R'-2(N_1+\cdots+N_4)-\Gamma))=8, \;\; \mbox{and} \\
  \label{eq:cubic2} h^0(\O_S(3R'-2(N_1+\cdots+N_4)-\Gamma-(N_5+\cdots+N_8))=1.
  \end{eqnarray}
\end{lemma}

\begin{proof}
  The first statement is proved as in Lemma \ref{lemma:Juve-15-4}. 
One similarly  proves that the divisor
$
\Gamma':=R'-N_1-\cdots-N_4
$
on $S$ is represented by an irreducible $(-2)$-curve, proving \eqref{eq:gamma2}. 
One easily verifies that the divisor in \eqref{eq:cubic2} is linearly equivalent to $\Gamma +\Gamma'$, and  $\Gamma \cdot \Gamma'=0$, proving \ref{eq:cubic2}.
This can also be used to prove that the divisors in  \eqref{eq:quadriche}, \eqref{eq:cubic0} and \eqref{eq:cubic} are big and nef and thus compute their cohomology.
\end{proof}

As a consequence, $\Gamma$ is mapped by $\varphi_{R'}$ isomorphically to a rational normal quartic $\Gamma\subset \PP^4$ that passes through the nodes $x_1,x_2,x_3,x_4\in \overline{S'}$, and has $N_5,N_6,N_7,N_8$ as bisecant lines. 

As in the previous section, the surface $\overline{S'}\subset \PP^4$ is the complete intersection of a quadric and a cubic and we will now show that the cubic can be chosen to be singular at the nodes of $\overline{S'}$. 

\begin{lemma} \label{lemma:incubichesing}
  The surface $\overline{S'}$ is contained in precisely a pencil of cubics singular at $x_1,x_2,x_3,x_4$ and in a unique quadric. Furthermore, there exist precisely a pencil of quadrics  $Q \subset\mathbb P^4$ and a net of cubics $Y\subset\mathbb P^4$ such that both $Q$ and $Y$ contain $\Gamma+N_5+\cdots+N_8$ and $Y$ is singular at $x_1,x_2,x_3,x_4$. 
\end{lemma}

\begin{proof}
As in the proof of Lemma \ref{lemma:incubichesing2}, there is a unique quadric $Q$ containing $\overline{S'}\subset \mathbb P^4$ and this is necessarily smooth if there is a cubic containing $\overline{S'}$ and singular at  $x_1,\ldots,x_4$.

We denote by $\mathfrak{n}: = x_1 + x_2 + x_3 + x_4$ the set of four nodes of $\overline{S'}$. In  particular, the linear system  $|\I^2_{\mathfrak{n}/\PP^4}(3))|$ of cubics singular at $\mathfrak n$ has dimension  at least $14$.  Any such cubic either contains $\overline{S'}$ or cuts out on $\overline{S'}$ a divisor whose strict transform in $S$ lies in the linear system $|3R'-2(N_1+N_2+N_3+N_4)|$, which has dimension $12$ by \eqref{eq:cubic0}. Hence, the linear system $\Sigma$ of cubics singular at $\mathfrak n$ and containing $\overline{S'}$ is at least a pencil. In particular, the quadric $Q$ containing $\overline{S'}$ is smooth  and $\Sigma|_Q$ is a point. As in the proof of Lemma \ref{lemma:incubichesing2}, one concludes that $\dim \Sigma=1$ (and thus $\dim |\I^2_{\mathfrak{n}/\PP^4}(3))|=14$) by considering the projection $r_Q:\Sigma \dashrightarrow \Sigma|_Q$, whose center consists of a unique point determined by the only reducible cubic $Q \cup H$ with $H$ being the hyperplane through $x_1,x_2,x_3,x_4$, since, by \eqref{eq:gamma2}, the points $x_1,x_2,x_3,x_4$ do not lie in a plane.   This finishes the proof of the first assertion.

Similarly, one proves that the linear system  of quadrics in $\PP^4$ containing  $\Gamma+N_5+\cdots+N_8$ is a pencil, by noting that  any such quadric different from $Q$ intersects $\overline{S'}$ along a divisor whose strict transform on $S$ lies in the linear system $|2R'-\Gamma-(N_5+\cdots+N_8)|$, which has dimension zero by \eqref{eq:quadriche}. The fact that the linear system of cubics containing  $\Gamma+N_5+\cdots+N_8$ and singular at $x_1,x_2,x_3,x_4$ is $2$-dimensional
likewise follows from 
\eqref{eq:cubic2} and the first assertion.
\end{proof}

\begin{prop}\label{preparatoria4,4}
Let $\Gamma\subset\PP^4$ be a fixed rational normal curve. Take $4$ general points $x_1,x_2,x_3,x_4\in \Gamma$ and $4$ general bisecant lines $N_5,N_6,N_7,N_8$ to $\Gamma$.  Then the following hold:
\begin{itemize}
\item[(i)] There exist precisely a pencil of quadrics  $Q \subset\mathbb P^4$ and a net of cubics $Y\subset\mathbb P^4$ such that both $Q$ and $Y$ contain $\Gamma+N_5+N_6+N_7+N_8$ and $Y$ is singular at $x_1,x_2,x_3,x_4$.
\item[(ii)] For general $Q$ and $Y$ as in (i), the minimal desingularization $S$ of the complete intersection $\overline{S'}:= Q\cap Y$ is a marked non-standard Nikulin surface of genus $17$.
\item[(iii)] By varying $x_1,x_2,x_3,x_4\in \Gamma$ and $n_5,\ldots,n_8\in \Sym^2\Gamma$, one obtains all members  in a dense open subset of $\widehat{\F}_{17}^{\n,ns}$.\
\end{itemize}
\end{prop}

\begin{proof}
  The existence of precisely a pencil of quadrics $Q$ follows from Proposition \ref{razionale}.

  We set $\mathfrak{n}: = x_1 + x_2 + x_3 + x_4$ and observe that $h^0(\mathcal I^2_{\mathfrak{n}/\PP^4}(3))=15$ by generality. Restricting first to $\Gamma$ and then to $N_5+N_6+N_7+N_8$ (via maps $\rho$ and $r$ analogous to the maps \eqref{eq:rho4,4} and \eqref{eq:restr2} in the proof of Proposition \ref{preparatoria5,2}), one shows the existence of a family  of cubics  $Y$ as in (i) of dimension $\geq 2$. Equality holds, since it holds for nodes and lines (possibly in special position) as in Lemma \ref{lemma:incubichesing}. This proves (i), as well as the fact that a surface  $\overline{S'}:= Q\cap Y\subset \PP^4$ as in (ii) specializes to a projective model of a Nikulin surfaces, and has therefore trivial canonical bundle and  $4$ ordinary double points at $x_1,x_2,x_3,x_4$, while being smooth elsewhere. Let $\nu:S\to \overline{S'}$ be its minimal desingularization and $N_1,\ldots,N_4$ be the exceptional divisors. Define $R':=\nu^*\mathcal O_{\overline{S'}}(1)$ and $R:=\mathcal{O}_{S}(\Gamma+N_5+\cdots+N_8)$. Exactly as in the proof of Proposition \ref{ita}, one may show that the sum $N_1+\cdots +N_8$ is $2$-divisible and set $M:=(N_1+\cdots +N_8)/2$ and $H:=R+R'+M$, so that the $4$-tuple $(S,M,H,R)\in \widehat{\F}_{17}^{\n,ns}$. 
  
We now exploit the projective model $\overline{S'}$ of $S$ to write a factorization of the map $r$ similar to \eqref{fact}. Applying  Lemma \ref{lemma:incubichesing} and equation \eqref{eq:cubic}, one may easily conclude that $t=2$, which finishes the proof of (i).

% The restriction map $r$ in \eqref {eq:restr}  factors through the restriction of the sections to $\overline{S'}$:
%   \[ r: \Ker \rho \stackrel{r_{\overline{S'}}}{\longrightarrow}
%     \left(\Ker \rho\right)|_{\overline{S'}}  \stackrel{r'}{\longrightarrow} H^0(\O_{\PP^1}(1))^{\+ 4} \cong \CC^8. \]
% By Lemma \ref{lemma:incubichesing}, we have $\dim \Ker r_{\overline{S'}}=2$. At the same time,
% the sections in $\left(\Ker \rho\right)|_{\overline{S'}}$ induce (linearly independent) sections in $H^0(S,\O_S(3R'-2(N_1+N_2+N_3+N_4)-\Gamma)$, which has dimension $8$ by \eqref{eq:cubic}, whence, $\dim \left(\Ker \rho\right)|_{\overline{S'}} \leq 8$. Since, as we said above, $\dim \Ker \rho \geq 10$, it follows that $\dim \Ker \rho=10$ and $\dim \left(\Ker \rho\right)|_{\overline{S'}} =8$. Therefore,
%   \[ \left(\Ker \rho\right)|_{\overline{S'}}  \cong H^0(S,\O_S(3R'-2(N_1+N_2+N_3+N_4)-\Gamma) \cong \CC^8\]
% and $r'$ can be identified with the restriction map
%   \[ \CC^8 \cong H^0(S,\O_S(3R'-2(N_1+N_2+N_3+N_4)-\Gamma) \longrightarrow  H^0(\O_{N_1 \cup \dots \cup N_4}(1)) \cong \CC^8\]
%   on $S$. Hence
%   \[ \Ker r' \cong H^0(\O_S(3R'-2(N_1+\cdots+N_4)-\Gamma-(N_5+\cdots+N_8)))=\CC,\]
%   again by \eqref{eq:cubic2}, which means that $\coker r \cong \coker r' \cong \CC$. Consequently,
%   \[ \dim \Ker r =10-8+1=3,\]
%   which finishes the proof of (i).

Point (iii) follows from Lemma \ref{lemma:incubichesing}.
\end{proof}

\begin{proof}[Proof of Theorem \ref{thm:4,4}]
Fix a rational normal curve $\Gamma\subset\PP^4$. We proceed as in the proof of Theorem \ref{thm:5,2} andidentify $\Hilb^4({\PP^2})\times \Hilb^4(\PP^1)$ with the parameter space for pairs $(\mathfrak l,\mathfrak n)$, where $\mathfrak l$ is a set of four unordered  bisecant lines to $\Gamma$  and $\mathfrak{n}$ is a set of four unordered points on $\Gamma$. Over an open set of $\Hilb^4({\PP^2})\times \Hilb^4(\PP^1)$  we have a  $\PP^2$-bundle $\MM$ whose fibre over a point $\left(\mathfrak l,\mathfrak{n} \right)$ is the net (by Proposition \ref{preparatoria4,4}(i)) of cubics $Y$ containing
$\Gamma \cup \mathfrak l$ and singular at $\mathfrak{n}$. Finally, on
$\MM$ we have the $\PP^1$-bundle $\mathbb O$ whose fibre over a point $\left(\mathfrak{n},\mathfrak l,Y\right)\in\MM$ is the pencil (by Proposition \ref{preparatoria4,4}(i)) of quadrics containing
$\Gamma \cup \mathfrak l$.

 We have a natural action of $\PGL(2)$ on all our parameter spaces and we will now show that $\mathbb O\sslash\PGL(2)$ is rational. Denoting by $U\subset \Hilb^4({\PP^2})$ the open subset parametrizing $4$-tuples of points in general position and applying Kempf's descent lemma several times as in the proof of Theorem \ref{thm:5,2}, one reduces to showing that $U/\PGL(2)$ is rational. This holds true because $U/\PGL(3)$ is a point and 
 $
 \PGL(3)/\PGL(2)\simeq \PP^5$.

By Proposition \ref{preparatoria4,4}(iii), the $12$-dimensional rational quotient $\mathbb O\sslash\PGL(2)$ admits a dominant rational map 
$\mathbb O\sslash\PGL(2)\dashrightarrow \widehat{\F}_{17}^{\n,ns}$, whose fiber over a general $(S,M,H,R)\in \widehat{\F}_{17}^{\n,ns}$ is the pencil (by Lemma \ref{lemma:incubichesing}) of cubics containing the projective model $\overline{S'}\subset \PP^4$ of $S$ and singular at its nodes. As a consequence, the moduli space $\widehat{\F}_{17}^{\n,ns}$ is stably rational and the same holds true for the space $\widehat\P_{4}$. The space
${\F}_{17}^{\n,ns}$ is therefore unirational. 
To obtain the unirationality of $\R_{4,4}$, it is thus enough to show that $r_{4,4}$ is dominant. Equivalently, we need to check that for a general $(S,M,H,R,C)\in \widehat\P_{4}$, the fiber of $r_{4,4}$ over the point $(C,x_1+x_2+x_3+x_4,M^\vee|_C)\in \mathcal R_{g,4}$ is $2$-dimensional. By Lemma \ref{lemma:incubichesing} the projective model $\overline{S'}\subset \PP(H^0(S,R')^\vee)=\PP^4$ of $S$ is the complete intersection of a uniquely determined quadric and a cubic (moving in a pencil) singular at $x_1,x_2,x_3,x_4$. It is thus enough to prove the existence of precisely a pencil of quadrics  $Q \subset\mathbb P^4$ and a net of cubics $Y\subset\mathbb P^4$ with $Y$ singular at $x_1,x_2,x_3,x_4$ such that both $Q$ and $Y$ contain $C\subset \mathbb P(H^0(C,\omega_C\otimes M)^\vee)=\mathbb P^4$ and $4$ lines. To check this, specialize $C$ to the curve $\Gamma+N_5+\cdots+N_8\subset \overline{S'}$ and apply Proposition \ref{preparatoria4,4}(i). \color{black}
\end{proof}

\section{The cases $\R_{4,2}$ and $\R_{3,6}$}

We will prove the following theorem in this section:

\begin{thm}\label{brevetto}
The moduli spaces $\F_{15}^{\n,ns}$, $\P_4$ and $\P'_{3}$ are unirational, and the maps $r_{4,2}$, $r_{3,6}'$ are both dominant. In particular, the moduli spaces $\R_{4,2}$ and $\R_{3,6}$ are unirational.
\end{thm}
For a general $(S,M,H)\in \F_{15}^{\n,ns}$, the polarizations $R$ and $R'$ have genus $g=4$ and $g'=3$, respectively. In particular, the image of the morphism
$$\varphi_{R'}:S\longrightarrow \overline{S'}\subset \mathbb P^3$$
is a quartic $\overline{S'}$ containing $6$ lines $N_3,\ldots,N_8$ and having two double points $x_1,x_2$ arising from the contraction of $N_1,N_2$. We consider the divisor
$$
\Gamma:=R-N_3-N_4-N_5-N_6;
$$
as before one proves that it is represented by an irreducible ($-2$)-curve. Since $\Gamma\cdot R'=3$, the image of $\Gamma$ under $\varphi_{R'}$ is a twisted cubic. It is trivial to check that $\Gamma$ passes through the nodes $x_1,x_2$ and is disjoint from the lines $N_7,N_8$, while the lines $N_3,N_4,N_5,N_6$ are bisecant to $\Gamma$.

\begin{prop}\label{finale}
Let $\Gamma\subset \PP^3$ be a fixed twisted cubic. Take two general points $x_1,x_2\in \Gamma$, four general bisecant lines $N_3,N_4,N_5,N_6$ to $\Gamma$ and a general line $N_7\subset\PP^3$ disjoint from $\Gamma,N_3,N_4,N_5,N_6$. Then the following hold:
\begin{itemize}
\item[(i)] There exists a unique quartic $\overline{S'}\subset\mathbb P^3$ containing $\Gamma,N_3,N_4,N_5,N_6,N_7$ and singular at $x_1,x_2$.
\item[(ii)] The minimal desingularization $S$ of a quartic $\overline{S'}\subset\mathbb P^3$ as in (i) is a non-standard Nikulin surface of genus $15$.
\item[(iii)] By varying $x_1,x_2\in \Gamma$, $n_3,\ldots, n_6\in \Sym^2\Gamma$ and $N_7\in \GG(1,3)$, one obtains all Nikulin surfaces in a dense open subset of $\F_{15}^{\n,ns}$.
\end{itemize}
\end{prop}
\begin{proof}
Setting $\mathfrak n:=x_1+x_2$, we observe that $h^0(\mathcal I^2_{\mathfrak{n}/\PP^3}(4))=27$ and thus the kernel of the
  restriction map
  \begin{equation}\label{forza}
    \rho:  H^0(\mathcal I^2_{\mathfrak{n}/\PP^3}(4)) \longrightarrow
    H^0(\mathcal \O_{\Gamma}(4)(-2(x_1+x_2))\cong H^0(\O_{\PP^1}(8)) \cong \CC^9
    \end{equation}
has dimension $\geq 18$. We consider the restriction map to $N_3+N_4+N_5+N_6+N_7$
\begin{equation}\label{roma}
  r: \Ker \rho \longrightarrow H^0(\oplus_{i=3}^6 \O_{N_i}(4-2)\oplus \O_{N_7}(4))\cong \CC^{17},
\end{equation}
 whose kernel has dimension $\geq 1$. The quartics $\overline{S'}\subset\mathbb P^3$ as in (i) are parametrized by $\PP(\Ker r)$ and this proves the existence part in (i). To obtain uniqueness, by semicontinuity it is enough to specialize to the case where $x_1,x_2$ and $N_3,N_4,N_5,N_6,N_7$ are nodes and lines on the projective model $\overline{S'}$ of a Nikulin surface $S$ and use the fact that the line bundle 
 $$
 4R'-\Gamma-N_3-N_4-N_5-N_6-N_7-2N_1-2N_2 \sim 3R-2N_3-\cdots-2N_6-3N_7-2N_8
 $$
 is not effective (which can be proved using similar techniques as above). This also implies that general quartics as in (i) can be specialized to projective models of general Nikulin surfaces. 
 
 We now turn to point (ii). Being a quartic, $\overline{S'}$ has trivial canonical bundle and by generality (as this holds true  when $x_1,x_2,N_3,N_4,N_5,N_6,N_7$ lie the projective model of a Nikulin surface) its only singularities are $2$ nodes at $x_1,x_2$ whose inverse images  under the desingularization map $\nu:S\to\overline{S'}$ are two ($-2$)-curves $N_1,N_2$. Set $R':=\nu^*\mathcal O_{\overline{S'}}(1)$ and $N_8:=2\Gamma-2R'+N_1+\cdots +N_6-N_7$. As in the proof of Proposition \ref{preparatoria}, one shows that $N_8$ is an effective divisor satisfying $N_8^2=-2$ and $N_8\cdot N_j=0$ for $j\neq 8$, and the sum $\sum_{i=1}^8N_i$ is $2$-divisible in $\Pic(S)$. Setting $M:=(N_1+\cdots N_8)/2$, $R:=\nu^*\mathcal{O}_{ S}(\Gamma+N_3+\cdots+N_6)$ and $H:=R+R'+M$, in order to prove (ii) it only remains to verify that $N_8$ is irreducible. Again as in the proof of of Proposition \ref{preparatoria}, this automatically follows if one shows that the constructed $K3$ surfaces move in a family of dimension $\geq11$.
 
By (i) the family of $K3$ surfaces we have constructed depends on   
$$
 \dim\left(\Sym ^2(\Gamma)\times \Hilb^4(\mathbb P^2)\times \GG(1,3)\right))-\dim PGL(2)=14-3=11
 $$
 parameters, where we have used the correspondence between bisecant lines to $\Gamma$ and points in $\PP^2$ and the fact that any projectivity of $\mathbb P^3$ fixing $\Gamma$ maps a quartic $\overline{S'}$ to a projectively equivalent surface. This implies points (ii), (iii).
\end{proof}

\begin{proof}[Proof of Theorem \ref{brevetto}]
We fix a twisted cubic $\Gamma\subset \PP^3$ and identify $\Sym ^2(\Gamma)\simeq \PP^2$, so that the product $ \Hilb^4(\mathbb P^2)\times \Hilb^2(\mathbb P^1)\times \GG(1,3)$ parametrizes triples $(\mathfrak l,\mathfrak n,l_7)$, where  $\mathfrak l$ is a set of $4$ bisecant lines to $\Gamma$, $\mathfrak n$ is a set of $2$ unordered point on $\Gamma$,  and $l_7$ is a line in $\mathbb P^3$. The group $PGL(2)\subset PGL(3)$ of projectivities fixing $\Gamma$ acts on this parameter space and by Proposition \ref{finale} the quotient $\Hilb^4(\mathbb P^2)\times \Hilb^2(\mathbb P^1)\times \GG(1,3)\sslash\PGL(2)$ is birational to a cover $\widetilde \F$ of degree ${6\choose 4}\cdot {2\choose 1}=30$ (corresponding to the choices of $4$ out of $6$ lines and of $1$ out of the remaining $2$ lines) of $\F_{15}^{\n,ns}$. Let $U\subset \Hilb^4({\PP^2})$ be the open subset parametrizing $4$-tuple of points in general position so that $U/\PGL(2)\simeq \PP^5$. By applying Kempf's descent lemma several times as in the proof of Theorem \ref{thm:4,4}, we obtain that $\widetilde \F$ is rational  and thus the unirationality of $\F_{15}^{\n,ns}$ and $\P_{4}$. 

To prove that $r_{4,2}$ is dominant, we need to show that, given a general $(S,M,H,C)\in\P_{4}$, the curve $C\subset\PP(H^0(\omega_C\otimes M)^\vee)= \mathbb P^3$ is contained in at most a $5$-dimensional family of quartics that are singular at the points $x_1,x_2\in C$ and contain $6$ lines. We specialize $C$ to the curve $\Gamma+N_3+\cdots+ N_6\in |R|$ and consider the maps \eqref{forza} and \eqref{roma} in the proof of Proposition \ref{finale}. Let 
\begin{equation*}
  r': \Ker \rho \longrightarrow H^0(\oplus_{i=3}^6 \O_{N_i}(2))\cong \CC^{12},
\end{equation*}
 be the composition of \eqref{roma} with the natural projection to $H^0(\oplus_{i=3}^6 \O_{N_i}(2))$. As the unicity statement in Proposition \ref{finale}(i) implies both that $\Ker \rho$ is $18$-dimensional and that $r$ is surjective, we conclude that $r'$ is surjective as well and that $\PP(\Ker r')=\PP^5$. Since the latter space parametrizes quartics containing $\Gamma+N_3+\cdots+ N_6$ and singular at $x_1,x_2$, we conclude that $r_{4,2}$ is dominant and $\R_{4,2}$ is unirational.\color{black}
 
 In order to study $r_{3,6}'$, we map a general $(S,M,H)\in \F_{15}^{\n,ns}$ to $\PP^4$ by
$$\varphi_{R}:S\longrightarrow \overline{S}\subset \mathbb P^4,$$
so that the image $\overline{S}$ has $6$ nodes $x_3,\ldots, x_8$, contains $2$ lines $N_1,N_2$, and is the complete intersection of a quadric $Q$ and a cubic $Y$. Since a general curve in $|R|$ is not hyperelliptic by \cite[Prop. 5.6]{DL}, the quadric $Q$ has to be smooth.  Let $q(z)=0$ and $f(z)=0$ be the equations defining $Q$ and $Y$, respectively. By possibly replacing $Y$ with a cubic of equation $f(z)+l(z)q(z)=0$ for some linear form $l(z)$, we can assume that $Q$ and $Y$ are everywhere transverse. As a consequence, $Y$ has to be singular at the $6$ nodes of $\overline{S}$. The dominance of $r_{3,6}'$ follows once we show that, for a general $C'\in |R'|$, the curve $C'\subset \PP(H^0(\omega_{C'}\otimes M)^\vee)=\PP^4$ is contained in a $2$-dimensional family of complete intersections $Q_t\cap Y_t$ of a quadric and a cubic containing $2$ lines and such that $Y_t$ is singular at the $6$ marked points $x_3,\ldots,x_8\in C'$. The curve $C'$ is contained in a unique cubic (that thus coincides with $Y$) passing through $x_3,\ldots, x_8$. Indeed, any such cubic different from $Y$ would cut out on $\overline{S}$ a divisor whose strict transform in $S$ lies in the linear system $|3R-R'-N_3-\cdots-N_8|$, which is empty. Analogously, since 
$|2R-R'|$ is a pencil, $C'$ is contained in a net of quadrics. In conclusion, $C'$ is contained in a $2$-dimensional family of complete intersections $Q_t\cap Y_t$ as above, and this proves both the dominance of $r_{3,6}'$ and the unirationality of $\R_{3,6}$.
\end{proof}

\section{The case $\R_{3,4}$}

This section aims to prove the following result:

\begin{thm}\label{ale}
The moduli spaces $\widehat{\F}_{13}^{\n,ns}$, $\F_{13}^{\n,ns}$ and $\widehat\P_{3}$ are unirational, and the map $r_{3,4}$ is dominant. In particular, the moduli space $\R_{3,4}$ is unirational.
\end{thm}
Take a general $(S,M,H)\in \F_{13}^{\n,ns}$, so that $g(R)=g(R')=3$. The morphism
$$
\varphi_{R'}:S\longrightarrow \overline{S'}\subset\PP^4
$$
maps $S$ onto a quartic $\overline{S'}$ containing $4$ lines $N_5,\ldots,N_8$ and having $4$ nodes $x_1,\ldots,x_4$ at the points where the curves $N_1,\ldots,N_4$ are contracted. The divisor
$$\Gamma:=R-N_5-N_6-N_7$$
is an irreducible ($-2$)-curve on $S$, which is sent by $\varphi_{R'}$ to a twisted cubic  passing through $x_1,\ldots,x_4$; the lines $N_5,N_6,N_7$
are bisecant to $\Gamma$, while $N_8$ is disjoint from it.

\begin{prop}\label{finale2}
Let $\Gamma\subset \PP^3$ be a fixed twisted cubic. Take $4$ general points $x_1,\ldots,x_4\in \Gamma$ and $3$ general bisecant lines $N_5,N_6,N_7$ to $\Gamma$. Then the following hold:
\begin{itemize}
\item[(i)] There exists a $4$-dimensional family of quartics $\overline{S'}\subset\mathbb P^3$ containing $\Gamma,N_5,N_6,N_7$ and singular at $x_1,\ldots,x_4$.
\item[(ii)] The minimal desingularization $S$ of a quartic $\overline{S'}\subset\mathbb P^3$ as in (i) is a marked non-standard Nikulin surface of genus $13$.
\item[(iii)] By varying $x_1,\ldots,x_4\in \Gamma$ and $n_5, n_6,n_7\in \Sym^2\Gamma$, one obtains all members of a dense open subset of $\widehat{\F}_{13}^{\n,ns}$.
\end{itemize}
\end{prop}
\begin{proof}
Set $\mathfrak n:=x_1+\cdots+x_4$. By restricting $H^0(\mathcal I^2_{\mathfrak{n}/\PP^3}(4))$ first to $\Gamma$ and then further to $N_5+N_6+N_7$ as in the proof of Proposition \ref{finale}, one proves the existence of a family $\F$ of quartics $\overline{S'}$ as in (i) of dimension $t\geq 4$, which implies (i) if we show that $t=4$. This is proved by specialization to nodes and lines on the projective model of a Nikulin surface $S$, where the equality
$
\dim |4R'-\Gamma-N_5-N_6-N_7-2N_1-2N_2-2N_3-2N_4|=3
$
(which can be proved with the same techniques used several times above) implies the statement.
 Again as in the proof of Proposition \ref{finale}, one shows that the minimal desingularization $S$ of a quartic $\overline{S'}\subset\mathbb P^3$ in the family $\F$ is a non-standard Nikulin surface of genus $13$. Hence (ii) holds and, by varying $x_1,\ldots,x_4\in \Gamma$ and $n_5, n_6,n_7\in \Sym^2\Gamma$ and modding out by the automorphisms of $\Gamma$, one obtains a family of Nikulin surfaces of dimension $7+4= 11$, thus yielding (iii). 
 \end{proof}
 
 \begin{proof}[Proof of Theorem \ref{ale}]
As in the proof of Theorem \ref{brevetto}, we fix a twisted cubic $\Gamma\subset \PP^3$ and identify $ \Hilb^3(\mathbb P^2)\times \Hilb^4(\mathbb P^1)$ with the parameter space for pairs $(\mathfrak l,\mathfrak n)$ with $\mathfrak l$ a set of $3$ bisecant lines to $\Gamma$ and $\mathfrak n$ a set of $4$ unordered point on $\Gamma$. By Proposition \ref{finale2} the quotient $\Hilb^3(\mathbb P^2)\times \Hilb^4(\mathbb P^1)\sslash\PGL(2)$ is birational to a cover $\widetilde \F$ of degree $4$ (corresponding to the choices of $3$ out of $4$ lines) of $\widehat\F_{13}^{\n,ns}$. Again as in the proof of Theorem \ref{thm:4,4} one applies Kempf's descent lemma to conclude that $\widetilde \F$ is rational, so that $\widehat\F_{13}^{\n,ns}$, $\F_{13}^{\n,ns}$ and $\widehat\P_{3}$ are unirational. 

It remains to show that $r_{3,4}$ is dominant, or equivalently, that for general $(S,M,H,C)\in\widehat\P_{3}$, the curve $C\subset \PP(H^0(\omega_C\otimes M)^\vee)=\mathbb P^3$ is contained in a $4$-dimensional family of quartics that are singular at the points $x_1,\ldots,x_4\in C$ and contain $4$ lines. It is then enough to specialize $C$ to $\Gamma+N_5+N_6+ N_7\in |R|$ and apply Proposition \ref{finale2}(i).
 \end{proof}

\section{The cases  $\R_{3,2}$ and $\R_{2,6}$}

We will prove the following result:

\begin{thm}\label{enri}
The moduli spaces $\F_{11}^{\n,ns}$, $\P_{3}$ and $\P'_{2}$ are rational, and the maps $r_{3,2}$, $r_{2,6}'$ are both dominant. In particular, the moduli spaces $\R_{3,2}$ and $\R_{2,6}$ unirational.
\end{thm}

The rationality of $\F_{11}^{\n,ns}$ has already been proved in \cite[Thm.1.2]{KLV2}. However, in order to prove the dominance of the moduli maps we will provide alternative parametrizations of $\F_{11}^{\n,ns}$. We start with the case of $r_{2,6}'$ because of its similarity to the previous section. Indeed, we consider the morphism
$$\varphi_{R}:S\longrightarrow \overline{S}\subset \mathbb P^3,$$
whose image  $\overline{S}$ is a quartic containing $2$ lines $N_1,N_2$ and having $6$ double points $x_3,\ldots,x_8$ arising from the contraction of $N_3,\ldots,N_8$. The divisor
$$
\Gamma:=R'-N_1-N_2,
$$
is represented by an irreducible ($-2$)-curve satisfying $\Gamma\cdot R=3$, so that the image of $\Gamma$ under $\varphi_{R}$ is a twisted cubic passing through $x_3,\ldots,x_8$ and bisecant to the lines $N_1,N_2$.

\begin{prop}\label{finale1}
Fix a twisted cubic $\Gamma\subset \PP^3$. Take $6$ general points $x_3,\ldots,x_8\in \Gamma$ and $2$ general bisecant lines $N_1,N_2$ to $\Gamma$. Then the following hold:
\begin{itemize}
\item[(i)] There exists a $4$-dimensional family of quartics $\overline{S}\subset\mathbb P^3$ containing $\Gamma,N_1,N_2$ and singular at $x_3,\ldots,x_8$.
\item[(ii)] The minimal desingularization $S$ of a quartic $\overline{S}\subset\mathbb P^3$ as in (i) is a non-standard Nikulin surface of genus $11$.
\item[(iii)] By varying $x_3,\ldots,x_8\in \Gamma$, $n_1, n_2\in \Sym^2\Gamma$, one obtains all Nikulin surfaces in a dense open subset of $\F_{11}^{\n,ns}$.
\end{itemize}
\end{prop}
\begin{proof}
  Set $\mathfrak n:=x_3+\cdots+x_8$. Again as in the proof of Proposition \ref{finale}, by restricting $H^0(\mathcal I^2_{\mathfrak{n}/\PP^3}(4))$ first to $\Gamma$ and then further to $N_1+\cdots+N_4$, one proves the existence of a family $\F$ of quartics $\overline{S}$ as in (i) of dimension $t\geq 3$ and we will show that $t=4$. We specialize to nodes and lines on the projective model of a Nikulin surface $S$. In this case the strict transform in $S$ of the restriction to $\overline S$ of any quartic in $\F\setminus\{\overline S\}$ is a divisor in the linear system
  \begin{equation}\label{petra}|4R-\Gamma-N_1-N_2-2N_3-\cdots-2N_8|=|R'+2\Gamma|.\end{equation} One sees that $\Gamma$ is a base divisor of \eqref{petra} and that $R'+\Gamma$ ie nef with $(R'+\Gamma)^2=4$, whence this linear system has has dimension $3$. Hence  $t= 4$ in this case. Let $\Z^{\n}\subset \Sym^2(\PP^2)\times \Sym^6(\Gamma)$ denote the subscheme parametrizing nodes and lines on projective models of Nikulin surfaces. We should have 
 $$\dim \Z^{\n}+4-\dim\Aut(\Gamma)=\dim\F_{11}^{\n,ns}=11;$$
 hence $\Z^{\n}$ is dense in $\Sym^2(\PP^2)\times \Sym^6(\Gamma)$, yielding (i),(ii),(iii).
 that is,  and that two quartics are isomorphic if they differ by the action of $\mathrm{Aut}(\Gamma)$.
 \end{proof}

\begin{proof}[Proof of Theorem \ref{enri}]
The rationality of $\F_{11}^{\n,ns}$, $\P_{3}$ and $\P'_{2}$ follows from  \cite[Thm. 1.2]{KLV2} and can be alternatively obtained from Proposition \ref{finale1}(iii). 

In order to prove that $r_{2,6}'$ is dominant, it is enough to show that for a general $(S,M,H,C')\in \F_{11}^{\n,ns}$ there exists a curve $C'\in |R'|$ such that $(C',N_3+\cdots+N_8|_C,M^\vee|_C)$ is contained in a $4$-dimensional family of Nikulin surfaces. It is then enough to choose $C':=\Gamma+N_1+N_2$ and apply Proposition \ref{finale1}(ii).

As concerns $r_{3,2}$, we need to show that its fiber containing a general $(S,M,H,C)\in \P_{3}$ is $6$-dimensional, or equivalently, that the family of Nikulin surfaces containing $(C,N_1+N_2|_C,M^\vee|_C)$ has dimension $6$.  We use the rational parametrization of  $\F_{11}^{\n,ns}$ provided in \cite[\S 5]{KLV2}. Let $\gamma\subset \PP^3$ be a fixed twisted cubic and denote by $p:T\to \PP^3$  the blow-up of $\PP^3$ along $\gamma$ and by $P_\gamma=p^{-1}(\gamma)$ the exceptional divisor. The pair $(p^*\O_{\PP^3}(1), p^*\O_{\PP^3}(2)-P_\gamma)$ provides an embedding of the threefold $T$ in $\PP^2\times \PP^3$ so that $T=(\PP^2\times \PP^3)\cap \PP^9\subset \PP^{11}$ and $\omega_T\simeq \mathcal O_T(-1,-2)$, where $\PP^2\times \PP^3\subset\PP^{11}$ is the Segre embedding. Denoting by $p':T\to \PP^2$ the first projection, the intersection of the exceptional divisor $P_\gamma$ with $P_l:=p'^{-1}(l)$, where $l\subset\PP^2$ is any line, is a smooth rational curve $\Gamma_l\subset T$. By \cite[Lem. 5.3]{KLV2}, every non-standard Nikulin surface $S$ of genus $11$ can be realized as an element of $|I_{\Gamma_l/T}(1,2)|$ for some $l$ so that $\O_S(1,0)\simeq R'$,  $\O_S(0,1)\simeq R$ and $\Gamma_l+N_1+N_2\in |R'|$. Let$(S,M,H,C)\in \P_{3}$ be general and consider the embedding $ S\subset T\subset \PP^2\times \PP^3$ restricting to the embedding $C\subset \PP^2\times \PP^2$ defined by the linear systems $(|\omega_C\otimes \eta^\vee|, |\omega_C|)$ where $\eta^\vee=M|_C$. The line $l\subset \PP^2$ for which $\Gamma_l$ is contained in $S$ is the image under $p'$ of the marked points $x_1,x_2\in C$. The ideal sequence of $C\cup\Gamma_l\subset S\subset T$ twisted by $\O_T(1,2)$ is
$$
0\longrightarrow \O_T\longrightarrow I_{C\cup \Gamma_l/T}(1,2)\longrightarrow I_{C\cup \Gamma_l/S}(1,2)\longrightarrow 0.
$$
We compute that $I_{C\cup \Gamma_l/S}(1,2)\simeq R+N_1+N_2$ and $N_1,N_2$ are its base components. As a consequence, $h^0(S, I_{C\cup \Gamma_l/S}(1,2))=h^0(S,R)=4$ and $h^0(T, I_{C\cup \Gamma_l/T}(1,2))=5$. Having fixed $C\subset T$, all Nikulin surfaces containing the embedded curve $C$ also contain the curve $\Gamma_l$ and thus move in a family of dimension $4$. In order to compute the dimension of the fiber of $r_{3,2}$ over the point $(C,x_1+x_2,\eta)\in \R_{3,2}$, it remains to bound the dimension of the family $\E$ of all possible embeddings of $C$ in $T$ so that the line bundle $p^*\O_T(2)-P_\gamma$ restricts to $\omega_C\otimes \eta^\vee$ and $p^*\O_T(1)$ restricts to $\omega_C$. Any such embedding factors as $C\subset Y\subset T$, where $Y\in |\O_T(0,1)|$ is a Del Pezzo surface of degree $6$ and $p'|_Y:Y\to \PP^2$ is the blow-up of $\PP^2$ at the $3$ nodes $y_1,y_2,y_3$ of the plane quintic $p'(C)\subset \PP^2$. Setting $l:=\O_Y(1,0)$ and denoting by $E_1,E_2,E_3\subset Y$ the exceptional divisors, it turns out that $O_Y(0,1)\simeq 2l-E_1-E_2-E_3$. Up to the action of $\Aut(T)=\Aut(\gamma)$, we may assume that $Y\subset T$ is fixed and we reduce to control the possible embeddings $C\subset Y$ satisfying $C\sim 5l-2E_1-2E_2-2E_3$ and $l|_C\simeq \omega_C\otimes \eta^\vee$. Any two such embeddings differ by an automorphism of $Y$, or equivalently, by an automorphism of $\PP^2$ fixing $y_1,y_2,y_3$. We conclude that $\dim \E=2$ and thus $r_{3,2}^{-1}(C,x_1+x_2,\eta^\vee) $ has dimension $6$, as wanted.

\begin{comment}We specialize $C$ to the curve $\gamma+N_3+N_4+N_5\in |R|$ marked with the points $x_1,x_2\in \gamma$; this is mapped to $\mathfrak c+l_3+l_4+l_5\subset \PP^2$ by $\varphi_{R'}$, where $\mathfrak c$ is a conic and $l_3,l_4,l_5$ are lines. By Proposition \ref{stanca}, the family $\F$ of Nikulin surfaces containing the marked curve $\gamma+N_3+N_4+N_5\in |R|$ admits a map $h:\F\dashrightarrow \Sym^2((\PP^2)^\vee)$ with fiber over $l_6+l_7$ birational to the $2$-dimensional family of plane sextics totally tangent at $\mathfrak c,l_3,\ldots, l_7$ and singular at $x_1,x_2$. The family $\F$ is thus $6$-dimensional and this concludes the proof.\end{comment}
\end{proof}

\section{The case  $\R_{2,4}$}

This section aims to prove the following result:

\begin{thm}\label{giorgio}
  The moduli spaces  $\widehat{\F}_{9}^{\n,ns}$ and $\widehat\P_{2}$ are rational,
$\F_{9}^{\n,ns}$ is unirational, and 
the map $r_{2,4}$ is dominant. In particular, the moduli space $\R_{2,4}$ is unirational.
\end{thm}

The unirationality of $\F_{9}^{\n,ns}$ was proved in \cite[Thm.1.2]{KLV2}. More precisely, in \cite[Thm. 4.5]{KLV2} the rationality of a space parametrizing quadruples $(S,M,H,N_1+\cdots +N_4)$ with $(S,M,H) \in \F_{9}^{\n,ns}$ and $N_1+\cdots+N_4$ a suitable subset of the eight $(-2)$-curves in the Nikulin lattice was proved; one easily sees that this space is birational to $\widehat{\F}_{9}^{\n,ns}$. Thus, the rationality of
$\widehat{\F}_{9}^{\n,ns}$, whence of $\widehat\P_{2}$, is known. It remains to show that $r_{2,4}$ is dominant. 
\begin{proof}[Proof of Theorem \ref{giorgio}]
To prove dominance, we show that the fiber of $r_{2,4} $ containing a general $(S,M,H,C)\in\widehat\P_{2}$ has dimension $6$, or equivalently, the family of marked  non-standard Nikulin surfaces  of genus $9$ containing  $(C,N_1+\cdots +N_4|_C,M^\vee|_C)$ has dimension $6$. A Nikulin surface $S$ as above carries two divisors
\begin{equation*}
\gamma:=R-N_5-N_6,\,\,\gamma':=R-N_7-N_8,
\end{equation*}
which are irreducible ($-2$)-curves such that $\gamma+\gamma'\sim H-N_1-\cdots-N_8=:F$ is an elliptic pencil. We may thus specialize $C$ to the curve $\gamma+N_5+N_6\in |R|$ marked with the points $x_1,\ldots,x_4\in \gamma$, which are the intersection points with $N_1,\ldots, N_4$. We denote $\gamma+N_5+N_6$ by $X=X_1\cup_{y,y'}\PP^1\cup_{z,z'}\PP^1$ if we consider it as an abstract pointed semistable curve of genus $2$: the component $X_1\simeq\PP^1$ is marked with the four points $x_1,\ldots,x_4$ as well as with two pairs of points $x_5,x_5'$ and $x_6,x_6'$ where two copies of $\PP^1$ (called exceptional components in the sequel) are attached. Let $\F$ be the family of marked non-standard Nikulin surfaces containing the pointed curve $X$ in such a way that $X_1\in |R-N_5-N_6|$, $\{x_5,x_5'\}=X_1\cap N_5$, $\{x_6,x_6'\}=X_1\cap N_6$ and $\{x_i\}=X_1\cap N_1$ for $1\leq i\leq 4$. To show that $\F$ has dimension $6$, we exploit the construction provided in \cite[\S 4]{KLV2}. 

Let $T:=(\PP^2\times \PP^2)\cap\PP^7\subset\PP^8$ be a fixed Del Pezzo threefold and denote by $p'$ and $p$ the two projections; by \cite[Prop. 4.3]{KLV2}, a general genus $9$ Nikulin surface $S$ of non-standard type lives inside of $T$ as a divisor of type $(2,2)$ containing four vertical lines $N_1,\ldots, N_4$ and four horizontal lines $N_5,\ldots, N_8$, so that $\O_S(1,0)\simeq R'$,  $\O_S(0,1)\simeq R$. Since the lines $N_1,\ldots, N_4\subset \PP^8$ span $\PP^7$, up to the action of $\Aut(T)$ (which consists of the automorphisms of $\PP^2\times \PP^2$ fixing $\PP^7$) we may assume $N_1,\ldots, N_4$ to be fixed. The choice of $N_5,\ldots,N_8$ is then equivalent to the choice of the four points $n_5,\ldots, n_8\in \PP^2$ such that $n_i=p(N_i)$ for $5\leq i\leq 8$, or equivalently, of the four lines $l_i:=p'(N_i)\subset\PP^2$. 

First of all, we show that the starting embedding $\gamma+N_5+N_6\subset T$ is the only possible embedding of $X$ in $T$ such that $p'$ maps $X_1$ to a plane conic and the exceptional components to two lines, while $p'$ contracts the exceptional components and sends $X_1$ to a line. Setting $n_i=p'(N_i)\in\PP^2$ for $1\leq i\leq 4$ (these points are fixed by construction), there is a unique plane conic $\mathfrak c$ through $n_1,\ldots, n_4$ such that $(\mathfrak c,n_1+\cdots+n_4)$ is isomorphic to $(X_1,x_1+\cdots+x_4)$ as a pointed curve. Hence, for any embedding $X\subset T$ as above, we get $p'(X)=\mathfrak c+l_5+l_6$, where $l_5,l_6$ are the lines intersecting $\mathfrak c$ at the points $x_5,x_5'$ and $x_6,x_6'$, respectively. Since the lines $l_5,l_6$ uniquely determine two horizontal lines in $T$, in any embedding $X\subset T$ as above the exceptional components coincide with two fixed $N_5,N_6$. As a consequence, $p(X)$ coincides with the line $r$ through $n_5,n_6$. In conclusion, both $p'(X)$ and $p(X)$ are fixed and thus $X=(p'(X)\times p(X))\cap T\subset T$ is unique.   

The family $\F$ thus coincides with the family of marked non-standard Nikulin surfaces containing the embedded curve $\gamma+N_5+N_6\subset T$. If $N_7,N_8\subset T$ are two general horizontal lines, one may easily check that $\dim\,|I_{\gamma+N_1+\cdots+N_8/T(2,2)|}|=2$ by using the short exact sequence 
$$
0\longrightarrow \O_T\longrightarrow I_{\gamma+N_1+\cdots+N_8/T}(2,2)\longrightarrow I_{\gamma+N_1+\cdots+N_8/S}(2,2)\longrightarrow 0
$$
and the fact that $I_{\gamma+N_1+\cdots+N_8/S}(2,2)\simeq F+\gamma'$ has $\gamma'$ as base component. Hence the natural map $\F\to \Sym^2(\mathbb P^2)$, which sends $(S,M,H)$ to $(N_7,N_8)$, has $2$-dimensional fibers; this yields $\dim\F=6$, as wanted. \color{black}
\end{proof}
\color{black}

\appendix

\section{Irreducibility of $\widehat{\F}_h^{\n,ns}$} \label{appendix}

%For $h \equiv 1 \; \mod 4$ let 
%$\widehat{\F}_h^{\n,ns}$ be the moduli space of isomorphism classes of pairs $[(S,M,H),R]$ such that $(S,M,H) \in \F_h^{\n,ns}$ and $R \in \Pic S$ such that $H-2R$ is the sum of four of the eight $(-2)$-curves in the Nikulin lattice.

The forgetful double cover $\widehat{\F}_h^{\n,ns} \to \F_h^{\n,ns}$ is  unramified over the proper closed locus $(\F_h^{\n,ns})^{\tiny{\mbox{aut}}}$ parametrizing nonstandard Nikulin surfaces $S$ possessing an automorphism exchanging $R$ and $R'$. In particular, $\widehat{\F}_h^{\n,ns}$ is
smooth except possibly over $(\F_h^{\n,ns})^{\tiny{\mbox{aut}}}$.

We will prove:

\begin{prop} \label{prop:app}
The moduli space $\widehat{\F}_h^{\n,ns}$ is irreducible for $h \geq 9$. 
\end{prop}

To prove the proposition it will be enough to exhibit a deformation of some $(S,M,H) \in \F_h^{\n,ns} \setminus (\F_h^{\n,ns})^{\tiny{\mbox{aut}}}$ carrying $R$ to $R'$.

We recall the smooth partial compactification $\overline{\F}_h^{\n,ns}$ of $\F_h^{\n,ns}$ constructed in \cite[Cor. 5.9]{KLV1} obtained by adding a smooth divisor parametrizing reducible surfaces as we now explain. (All details can be found in \cite[\S 5]{KLV1}.)

  Set $h=4k+5$, with $k \geq 1$.
Choose a smooth elliptic curve $A$, a general $x \in A$ and $L_1, L_2 \in \Pic^2(A)$ such that
\begin{equation}\label{eins}L_1^{\*2} \cong L_2^{\*2}.\end{equation} Then $L_1$ and $L_2$ provide an embedding $A\subset\PP^1 \x \PP^1$ such that $x$  uniquely determines a quadrilateral of lines $N_1+N_2+N_3+N_4$ in $\PP^1 \x \PP^1$, 
where $N_1, N_2$ and $N_3 , N_4$ respectively are in $\vert \mathcal O_{\PP^1 \x \PP^1}(1,0) \vert$ and $\vert \mathcal O_{\PP^1 \x \PP^1}(0, 1) \vert$, with $N_1 \neq N_2$ and $N_3 \neq N_4$. 
Denote the four singular points of the quadrilateral by
\[ e_{ij} := N_i \cap N_j, \;\; 1 \leq i \leq 2, \;\; 3 \leq j \leq 4.\]

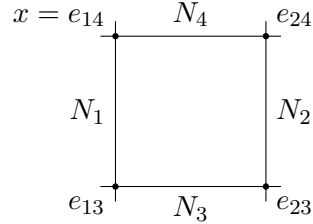
\begin{center}
\begin{tikzpicture}[scale=1]%,cap=round,>=latex]
\draw (-1.2,-1) -- node[below] {$N_3$} (1.2,-1);
\draw (-1,-1.2) -- node[left] {$N_1$} (-1,1.2);
\draw (-1.2,1) -- node[above] {$N_4$} (1.2,1);
\draw (1,-1.2) -- node[right] {$N_2$} (1,1.2);
\draw (-1,1)  node[above left] {$x=e_{14}$};
\draw (-1,-1)  node[below left] {$e_{13}$}; 
\draw (1,1)  node[above right] {$e_{24}$};
\draw (1,-1)  node[below right] {$e_{23}$}; 
\draw[fill] (-1,1) circle [radius=1pt];
\draw[fill] (-1,-1) circle [radius=1pt];
\draw[fill] (1,1) circle [radius=1pt];
\draw[fill] (1,-1) circle [radius=1pt];
\end{tikzpicture}
\end{center}

Let $X_4$
be the blowing up of $e_{14},e_{24}, e_{23}, e_{13}$ and denote by
$E_{ij} := \sigma^{-1}(e_{ij})$ the four exceptional curves on $X_4$. By abuse of notation, we still denote the strict transform of $N_j$ on $X_4$ by $N_j$. 
These are four disjoint smooth, rational curves of self-intersection $-2$.

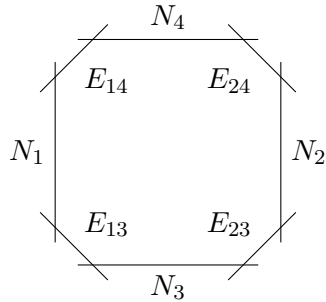
\begin{center}
\begin{tikzpicture}[scale=1],cap=round,>=latex]
\draw (-1.2,-1.5) -- node[below] {$N_3$} (1.2,-1.5);
\draw (-1.5,-1.2) -- node[left] {$N_1$} (-1.5,1.2);
\draw (-1.2,1.5) -- node[above] {$N_4$} (1.2,1.5);
\draw (1.5,-1.2) -- node[right] {$N_2$} (1.5,1.2);
\draw (-1.7,0.8)  -- node[below right] {$E_{14}$} (-0.8,1.7);
\draw (-1.7,-0.8) --  node[above right] {$E_{13}$} (-0.8,-1.7) ; 
\draw (0.8,1.7)  -- node[below left] {$E_{24}$} (1.7,0.8);
\draw (0.8,-1.7)  -- node[above left] {$E_{23}$} (1.7,-0.8);
\end{tikzpicture}
\end{center}

We make the same construction starting from the same $A$ but a (possibly) different point $x'=e_{14}'\in A$ satisfying
\begin{equation}\label{drei}
    ke_{14} \sim ke_{14}' \; 
\;  \mbox{on} \; \; A 
\end{equation}
and $L'_1, L'_2 \in \Pic^2(A)$ such that
\begin{equation}\label{zwei}{L'}_1^{\*2} \cong {L'}_2^{\*2},\end{equation}
thus obtaining another surface $X'_4$. We use the notation $e_{ij}'$, 
$N'_i$ and  $E_{ij}'$ for the objects analogous to $e_{ij}$, 
$N_i$ and  $E_{ij}$. 
Now choose other $8$ points $u_{1},\ldots,u_{4}, u_{1}',\ldots,u_{4}'$ on $A$ 
 satisfying:
\[
u_1+u_2+u_3+u_4+u_1'+u_2'+u_3'+u_4'\in |L_1^{\otimes2}\otimes L_2^{\otimes 2}|
\]
and denote by $X$ (respectively, $X'$) the blow up of $X_4$ (resp., $X'_4$) at $u_{1},\ldots,u_{4}$ (resp. $u_{1}',\ldots,u_{4}'$). By abuse of notation, we denote the strict transform on $X$ (or $X'$) of a divisor on $X_4$ (or $X_4'$) still by the same name. The surface $S:= X \sqcup_{A} X'$ obtained by gluing $X$ and $X'$ transversally along $A$ is a flat limit of $K3$ surfaces. 
One easily verifies (cf. \cite[Lemma 5.1]{KLV1}) that the divisors
\[ \Delta:= (k+1)N_1+N_2+N_3+(k+1)N_4+(2k+1)E_{14}+E_{13}+E_{24}+E_{23} \]
on $X$ and
\[ \Delta':=(k+1)N'_1+N'_2+N'_3+(k+1)N'_4+(2k+1)E'_{14}+E'_{13}+E'_{24}+E'_{23}  \]
 on $X'$ glue to form a Cartier divisor $H$ on $S$ (cf. \cite[(35)]{KLV1}).

One shows that $S$ also carries the line bundle $M:=\frac{1}{2}\left(N_1+\cdots+N_4+N'_1+\cdots+N'_4\right)$.  and two line bundles $R$ and $R'$ defined as follows
\begin{equation}
  \label{eq:keven}
  R:=\frac{1}{2}\left(H-N_1-N_2-N'_1-N'_2\right), \;  \; R':=\frac{1}{2}\left(H-N_3-N_4-N'_3-N'_4\right)
  \end{equation}
 if $k$ is even, and  
 \begin{equation}\label{eq:kodd}
  R:=\frac{1}{2}\left(H-N_1-N_3-N'_1-N'_3\right), \;  \; R':=\frac{1}{2}\left(H-N_2-N_4-N'_2-N'_4\right)
  \end{equation}
  if $k$ is odd (cf. \cite[p. 26]{KLV1}).

As proved in \cite[Cor. 5.9 and Prop. 5.4]{KLV1}, 
triples $(S,M,H)$ as above form a boundary divisor of a partial compactification $\overline{\F}_h^{\n,ns}$ of $\F_h^{\n,ns}$. Indeed, the construction of $S$
depends on $10$ parameters: one for the choice of $A$, one each for the choices of $x$, $L_1$ and $L'_1$ (which determine the choices of $x'$, $L_2$ and $L'_2$ in finitely many ways), minus one for the automorphisms of $A$, and then plus seven for $|L_1^{\otimes2}\otimes L_2^{\otimes 2}| \cong \PP^7$. While constructing $S$, we have also made a finite number of choices due to the relations \eqref{eins}-\eqref{zwei}; different choices correspond to different components of the obtained boundary divisor. We will henceforth consider a component $\mathfrak D$ where the relation 
\begin{equation}\label{vier}
L_1\otimes L_2^\vee\simeq L_1'\otimes {L'}_2^\vee
\end{equation}
holds. One may easily check that a general member of $\mathfrak D$ has no automorphisms. To prove the proposition, we will show that for a general $(S,M,H) \in \mathfrak D$ the pair $((S,M,H),R)$ can be deformed in $((S,M,H),R')$ remaining inside of $\mathfrak D$.

\begin{proof}[Proof of Proposition \ref{prop:app}]\
In the case $k$ even, we may act on $\Pic^2(A)\times \Pic^2(A)$ with $\Pic^0(A)$ as $\eta\cdot (\A_1,\A_2):=(\A_1\otimes \eta, \A_2\otimes \eta)$. Since $L_1^{\otimes 2}\simeq L_2^{\otimes 2}$, the pairs $(L_1,L_2)$ and $(L_2,L_1)$ are in the same orbit. This provides a one-parameter family of deformations of $X$ parametrized by $\Pic^0(A)$ so that two fibers of the family are $X$ itself but $(N_1,N_2)$ and $(N_3,N_4)$ are exchanged. Analogously, one obtains a  one-parameter family of deformations of $X'$ that plays the same game. By gluing, one obtains a two-parameter family of surfaces in $\mathfrak D$ so that two fibers of the family are the same $(S,M,H)$ but the divisors $R$ and $R'$ in \eqref{eq:keven} switch roles.

In the case $k$ odd, we may act on $A\times A$ with $A$ by $z\cdot (y_1,y_2):=(y_1\oplus z,y_2\oplus z)$, where the symbol $\oplus$ is the group operation on $A$. Since $e_{14}\ominus e_{23}$ is $2$-torsion, the pairs $(e_{14}, e_{23})$ and $(e_{23},e_{14})$ lie in the same orbit. This provides a one-parameter family of deformations of $X$ parametrized by $A$ so that two fibers of the family are $X$ itself but $e_{1,4}$ and $e_{2,3}$ are exchanged; as consequence, $(N_1,N_3)$ and $(N_2,N_4)$ are also exchanged. Since \eqref{drei} remains valid if we translate both $e_{1,4}$ and $e_{1,4}'$ by the same point, we can compatibly deform $X'$ thus obtaining a one-parameter deformation of $S=X\cup X'$. Relation \eqref{vier} yields $e_{14}'\ominus e_{23}'=e_{14}\ominus e_{23}$, which guarantees that the same translation interchanging $e_{1,4}$ and $e_{2,3}$ also swaps $e'_{1,4}$ and $e'_{2,3}$. We have thus obtained a one-parameter family of surfaces in $\mathfrak D$ with two fibers both isomorphic to $(S,M,H)$ but such that $R$ and $R'$ in \eqref{eq:kodd}  have switched roles. This concludes the proof of the proposition.
\end{proof}

  \end{document}